\documentclass[12pt,a4paper]{amsart}
\usepackage{etex}
\usepackage{amssymb}
\usepackage[T1]{fontenc}
\usepackage[latin1]{inputenc}
\usepackage{mathrsfs}
\usepackage{amsfonts,amsmath,lmodern,fancyhdr,lastpage,graphicx,nccfoots,caption}
\usepackage{afterpage}
\usepackage{tikz}
\usepackage{pgf,tikz}
\usepackage{placeins}
\usepackage{amsthm}
\usepackage{graphicx}
\usepackage{latexsym}
\usepackage{amsmath}
\usepackage{amsfonts}
\usepackage{amssymb}
\usepackage{pstricks-add}
\usepackage{slashbox}
\usepackage{url}
\usepackage{multicol}
\usepackage{xcolor}

\newcounter{a}
\ifodd\thea\else\stepcounter{a}\fi

\newcommand{\cchi}{\mbox{\raise.48ex\hbox{\,$\chi$}}}

\newcommand{\R}{\mathbb{R}}

\newcommand{\N}{\mathbb N}
\newcommand{\Z}{\mathbb Z}

\newcommand{\YT}[3]{
	\vcenter{\hbox{
			\begin{tikzpicture}[x={(0in,-#1)},y={(#1,0in)}] 
			\foreach \rowi [count=\i] in {#3} {
				\foreach \e [count=\j] in \rowi {
					\draw (\i,\j) rectangle +(-1,-1);
					\draw (\i-0.5,\j-0.5) node {$#2\e$};
				}
			}
			\end{tikzpicture}
		}}
	}

\newcommand{\SYT}[3]{
	\vcenter{\hbox{
			\begin{tikzpicture}[x={(0in,-#1)},y={(#1,0in)}] 
			\foreach \rowi [count=\i] in {#3} {
				\foreach \e [count=\j] in \rowi {
					\draw (\i,-\j) rectangle +(-1,-1);
					\draw (\i-0.5,-\j-0.5) node {$#2\e$};
				}
			}
			\end{tikzpicture}
		}}
	}

\newtheorem*{question*}{Question}
\newtheorem{theorem}{Theorem}
\newtheorem{corollary}[theorem]{Corollary}

\newtheorem{lemma}[theorem]{Lemma}
\newtheorem{proposition}[theorem]{Proposition}
\newtheorem*{remark}{Remark}
\newtheorem{example}[theorem]{Example}

\theoremstyle{definition}
\newtheorem{algorithm}[theorem]{Algorithm}
\newtheorem{definition}[theorem]{Definition}

\title{Symplectic keys and Demazure atoms in type $C$}
\author[J. M. Santos]{João Miguel Santos}
\address{CMUC, Department of Mathematics, University of Coimbra, Apartado 3008,
	3001--454 Coimbra, Portugal}\email{jmsantos@mat.uc.pt}

\begin{document}
	\thispagestyle{empty}

\begin{abstract}
	We compute, mimicking the Lascoux-Schützenberger type $A$ combinatorial  procedure, left and right keys for a Kashiwara-Nakashima tableau in type $C$. These symplectic keys have a similar role as the keys for semistandard Young tableaux. More precisely, our symplectic keys give a tableau criterion for the Bruhat order on the hyperoctahedral group and cosets, and describe Demazure atoms and characters in type $C$. The right and the left symplectic keys are related through the Lusztig involution. A type $C$ Schützenberger evacuation is defined to realize that involution.
\end{abstract}
\maketitle
\section{Introduction}

The irreducible characters of the general linear group $GL_n$, over $\mathbb{C}$, the Schur polynomials,
are combinatorially expressed as sums on semistandard Young tableaux with entries $\leq n$ \cite{You 52}. When restricting
to the symplectic group $Sp_{2n}$, two different types of symplectic tableaux have been proposed.  King showed that the irreducible symplectic characters, the symplectic Schur polynomials,  can be seen as a sum on a family of tableaux that are known as King tableaux \cite{King 75}, and De Concini has proposed    the ones known as De Concini tableaux \cite{DeC 79}.
Kashiwara and Nakashima \cite{Kash 95} described symplectic tableaux, which are just a variation of  De Concini tableaux, with a crystal graph structure. That crystal structure allows a plactic monoid compatible with insertion and sliding algorithms, and Robinson-Schensted type correspondence, studied by Lecouvey in terms of crystal isomorphisms \cite{Lec 02, Lec 07}. The generalization of the notion of plactic monoid for finite Cartan types was first introduced by Littelmann using his path model \cite{Litt 96}. Symplectic  Kashiwara-Nakashima tableaux are the ones that we work with, in the  corresponding ambient plactic monoid. We  however note that very recently  Lee    has endowed  King tableaux with a crystal structure \cite{Lee 19}.

Kashiwara \cite{Kash 93} and Littelmann \cite{ Litt 95} have shown that Demazure characters \cite{Dem 74}, for any Weyl group, can be lifted to certain subsets of the crystal $\mathfrak{B}^\lambda$ for a given dominant weight $\lambda$, a normal crystal with highest weight $\lambda$ \cite{BSch 17}, called Demazure crystal.
That is, a Demazure character (also known as key polynomials) is the generating function of the weights over a Demazure crystal. In type $C_n$, we consider $\mathfrak{B}^\lambda$ to be the crystal of $C_n$-Kashiwara-Nakashima tableaux of shape $\lambda$, and Demazure characters are  indexed by integer vectors in the orbit of the partition $\lambda$ under the action of the Weyl group, the hyperoctahedral group $B_n$. They are certain non symmetric Laurent polynomials, with respect to the action of the Weyl group, which can be seen as \emph{partial} symplectic characters, {\em i.e.},   sums of  a certain  portion  of  monomials in a symplectic Schur polynomial.

In type $A_{n-1}$, the Demazure crystals are certain subsets of the crystal $\mathfrak{B}^\lambda$, the crystal of all  semistandard Young tableaux of shape $\lambda$, with entries $\leq n$. Lascoux and Sch\"{u}tzenberger \cite{LasSchu 90} identified the tableaux with nested columns as key tableaux, and defined the right key map that sends tableaux to key tableaux. Their right key map gives a decomposition  of  $\mathfrak{B}^\lambda$ into non intersecting subsets $\mathfrak{U}(v)$, each containing a unique key, in bijection with the vectors $v$ in the orbit of $\lambda$, under the action of the Weyl  group, $\mathfrak{S}_n$ \cite[Theorem 3.8]{LasSchu 90}.
They called \emph{standard bases} the sum of monomial weights over $\mathfrak{U}(v)$, which, after Mason \cite{Mas 09}, are coined Demazure atoms.
The decomposition describes what  tableaux contribute to the Demazure crystal $\mathfrak{B}_v$, 
as a union of Demazure atoms,  over an interval in the Bruhat order, on the classes modulo the stabilizer of $\lambda$. This order, induced on the orbit of $\lambda$, gives $\mathfrak{B}_v=\displaystyle \biguplus_{\lambda\le u\le v} \mathfrak{U}(u)$. 


Our work has been motivated by the questions raised  in a presentation by Azenhas \cite{AzM 12}, in {\em The 69th Séminaire Lotharingien de Combinatoire}. In those questions, Azenhas identified some type $C_n$ Kashiwara-Nakashima tableaux as key tableaux, which match our identification, but it lacks a construction of the right key map, thus lacking a proof of the combinatorial description of type $C$ Demazure characters. Note also that, during the preparation of this paper, Jacon and Lecouvey informed us about their paper \cite{JaLec 20}, where they find the same key in type $C$, but their  approach is different from ours.

Inspired by the Lascoux-Schützenberger's construction of the left and right keys  of a semistandard Young tableau \cite{LasSchu 90}, we give a similar construction in type $C_n$. Our construction of the left and right keys of a Kashiwara-Nakashima tableau, in type $C_n$, is based on frank words in type $C$, that we introduce in Section \ref{SecKeys}, and Sheats symplectic \emph{jeu de taquin}. Our Theorem \ref{DemKey} is the type $C$ analogue of \cite[Theorem $3.8$]{LasSchu 90}. We also show, in Section \ref{SecLusztig}, that both keys, left and right, are related via the Schützenberger involution in type $C$, or Lusztig involution, realized here in an explicit way, using symplectic insertion or sliding operations. 

In \cite{Len 07}, using the model of alcove paths, Lenart defined an initial key and a final key, for any Lie type, related via the Lusztig involution,  which, in type $C$, have a similar behaviour to the left and right keys defined here. There is a crystal isomorphism between the alcove path model and the Kashiwara-Nakashima tableau  model in types $A$ and $C$ \cite{Len 12,LenLub 15}.  Since right an left keys in type $C$ are explicitly related through the Sch\"{u}tzenberger
involution in type $C$, or Lusztig involution, the left and right key maps in types $A$ and $C$ coincide in the aforesaid approaches or models.

The paper is organized as follows. In Section \ref{SecBruhat}, we discuss the Weyl group of type $C$, the signed permutation group $B_n$, the Bruhat order on $B_n$ and on its cosets, modulo the stabilizer of $\lambda$, the Kashiwara-Nakashima tableaux and the symplectic key tableaux. Those key tableaux are used in Proposition \ref{minimalBn} to explicitly construct the minimal length coset representatives.
We recall some results from Bjorner and Brenti's book  \cite{BB05} and Proctor \cite{Pro 81}, that lead to a tableau criterion, in theorems \ref{BruhatKey} and \ref{keycoset}, for the Bruhat order on $B_n$ and on its cosets, using the symplectic key tableaux. In Section \ref{SecRS}, we recall the Baker-Lecouvey insertion, the Sheats symplectic \textit{jeu de taquin} and use them to discuss the plactic and coplactic equivalences and the Robinson-Schensted type $C$ correspondence. These equivalences have a natural interpretation in the type $C_n$ Kashiwara crystal, for a $U_q(sp_{2n})$-module, in terms of connected components and crystal isomorphic connected components.
In Section \ref{SecKeys}, we extend the concept of frank word, in type $A$, to type $C$ and, with the help of symplectic \textit{jeu de taquin}, we present, in Theorem \ref{rightkeymap}, our  right and left key maps. Using the right key map, we describe the tableaux that contribute to a Demazure atom and to a Demazure crystal, which is our main result, Theorem \ref{DemKey}. In Section \ref{SecLusztig}, we develop a type $C$ evacuation within the plactic monoid, an analogue of the $J$-operation discussed by Schützenberger for semistandard Young tableaux in \cite{Schu 76}. This is an explicit realization of Lusztig involution using insertion and sliding operations in type $C$. Proposition \ref{key&Lusz} shows that the evacuation of the right key of a Kashiwara-Nakashima tableau is the left key of the evacuation of the same tableau. 

Note: An extended abstract of part of this work was accepted in the conference FPSAC 2020 \cite{San20}.

\section{Weyl group of type C, Bruhat order and symplectic key \\ tableaux}\label{SecBruhat}

Fix $n\in \N_{>0}$. Define the sets $[n]=\{1, \dots, n\}$ and $[\pm n]=\{1,\dots, n, \overline{n}, \dots, \overline{1}\}$ where $\overline{i}$ is just another way of writing $-i$. In the second set we will consider the following order of its elements: $1<\dots< n< \overline{n}< \dots< \overline{1}$ instead of the usual order.

Consider the group $B_n$, with generators $s_i,\, 1\leq i\leq n$, having the following presentation, regarding the relations among the generators,
\begin{align*}
	B_n := \langle s_1, \dots, s_{n}\mid&s_i^2 = 1, \, 1 \le i \le n;\; (s_is_{i+1})^3 = 1, 1 \le i \le n - 2;\;(s_{n-1}s_n)^4=1;\\&(s_is_j)^2 = 1, \,1 \le i < j \le n,\, |i - j| > 1\rangle,
\end{align*}
known as hyperoctahedral group or signed symmetric group. This group is a Coxeter group and we consider the (strong) Bruhat order on its elements \cite{BB05}. 
The elements of $B_n$ can be seen as odd bijective maps from $[\pm n]$ to itself, i.e., for all $\sigma\in B_n$ we have $\sigma(i)=-\sigma(-i),\, i\in [\pm n]$. The subgroup with the generators $s_1, \dots, s_{n-1}$ is the symmetric group $\mathfrak{S}_n$ and its elements can be seen as bijections from $[n]$ to itself. Both groups can also be seen as groups of $n\times n$ matrices. The elements of the symmetric group can be identified with the permutation matrices, and if we allow the non-zero entries to be either $1$ or $-1$, we have the elements of $B_n$. Hence $B_n$ has $2^nn!$ elements. The groups $\mathfrak{S}_n$ and $B_n$ are the Weyl groups for the root systems of types $A_{n-1}$ and $C_n$, respectively.

Let $\sigma\in B_n$. 
We denote $[a_1\, a_2\,\dots \,a_n]$, where $a_i=\sigma(i)$ for $i \in [n]$, the window notation of $\sigma$, and write $\sigma=[a_1\,a_2\,\dots\,a_n]$. 
The elements of $B_n$, or $\mathfrak{S}_n$, 
act on vectors in $\Z^n$ on the left. Given a vector $v\in \Z^n$, we have that $s_i$, with $i\in [n-1]$, acts on $v$ swapping the $i$-th and the $(i+1)$-th entries, and $s_n$ acts on $v$, $s_nv$, changing the sign of the last entry. Note that the window notation of $\sigma s_i $ is obtained after applying $s_i$ to the window notation of $\sigma$, if we see it as a vector. 
Ignoring signs, $\sigma v  =(v_{\sigma^{-1}(1)}, \dots, v_{\sigma^{-1}(n)})$, with $v=(v_1, \dots, v_n)$. The $i$-th letter of $\sigma v $ changes its sign if and only if $\overline{i}$ appears in $\sigma$. Hence $\sigma v  =(sgn(\sigma^{-1}(1)) v_{\left|\sigma^{-1}(1)\right|}, \dots, sgn(\sigma^{-1}(n))v_{\left|\sigma^{-1}(n)\right|})$, where $sgn(x)=1$ if $x$ is positive and $-1$ if $x$ is negative, for $x\in [\pm n]$.
\begin{example}
	Consider $v=(1,2,3)\in \Z^3$ and  $\sigma=[2\,\overline{3}\,1]=[s_1s_3s_2(1), s_1s_3s_2(2),s_1s_3s_2(3)]$ $=s_1s_3s_2$ $\in B_3$. So  \begin{align*}
		\sigma(1,2,3)&=s_1s_3s_2(1,2,3)=s_1s_3(1,3,2)=s_1(1,3,\overline{2})=(3,1,\overline{2})\\
		&= (sgn(\sigma^{-1}(1)) v_{\left|\sigma^{-1}(1)\right|}, sgn(\sigma^{-1}(2))v_{\left|\sigma^{-1}(2)\right|}, sgn(\sigma^{-1}(3))v_{\left|\sigma^{-1}(3)\right|})\\&=(1\cdot 3, 1\cdot 1, -1\cdot 2).
	\end{align*}
\end{example}
%

\subsection{Bruhat order on $B_n$}
The \emph{length} of $\sigma\in B_n$, $\ell(\sigma)$, is the least number of generators of $B_n$ needed to go from $[1\,2\,\dots\,n]$, the identity map, to $\sigma$. Any expression of $\sigma$ as a product of $\ell(\sigma)$ generators of $B_n$ is called reduced. We say that two letters of the window notation of $\sigma$ form an inversion if the bigger letter appears first. 
%
The next proposition gives a way to compute $\ell(\sigma)$ that only requires to look at the window notation of $\sigma$. This is a variation of the length formula presented on \cite[Proposition 8.1.1]{BB05}, where the authors consider the usual ordering of the alphabet $[\pm n]$ and the generator that changes the sign of an entry of the window notation acts on the first entry instead of the last one.

\begin{proposition}\label{length} Consider $\sigma\in B_n$. Then
	\begin{align*}
		\ell(\sigma)=\#\{\text{inversions of $\sigma$}\}+\sum\limits_{\text{$\overline{i}$ appears in $\sigma$}}(n+1-i).
	\end{align*}
\end{proposition}

The (signed) permutation $\sigma=[2\,\overline{3}\,1]$ has two inversions: $2,\,1$ and $\overline{3},\,1$ and $\ell(\sigma)=3$.
\pagebreak
\begin{remark}\indent
	\begin{itemize}
		\item If $\overline{i}$ does not appear in the window presentation of $\sigma$, for all $i\in [n]$, we may identify $\sigma$, in one-line notation, with $\sigma(1)\dots\sigma(n) \in \mathfrak{S}_n$ and $\ell(\sigma)=\#\{\text{inversions of }\sigma\}$ \cite[Proposition 1.5.2]{BB05}.
		\item We have that $\ell(\sigma s_i)>\ell(\sigma)$ if $i=n$ and $\sigma(n)$ is positive, or, $i\neq n$ and $\sigma(i)<\sigma(i+1)$.
	\end{itemize}
\end{remark}
The Bruhat order on the set of the elements of $B_n$ can be defined in the following way:
\begin{definition}\cite{BB05}\label{BOdef}
	Let $w=\sigma_1\dots \sigma_{\ell(w)}$, where $\sigma_i\in \{s_1,\dots,s_n\}$ are generators of $B_n$, and $u$ be two elements in $B_n$. Then $u\leq w$ in the Bruhat order if
	\begin{align*}
		\exists 1\leq i_1<i_2\dots<i_{\ell(u)}\leq \ell(w) \,\text{such that } u=\sigma_{i_1}\sigma_{i_2}\dots \sigma_{i_{\ell(u)}}.
	\end{align*}
\end{definition}
By definition, if $u\leq w$ then $\ell(u)\leq \ell(w)$, but the reverse is not true. If $\sigma(n)$ is positive and $i=n$, or, $\sigma(i)<\sigma(i+1)$ and $i\neq n$, we can also say that $\sigma s_i>\sigma$. 

The combinatorics of crystal graphs in type $C$ and the Bruhat order combinatorics on $B_n$ and cosets are  strongly related.
In subsections \ref{keyBruhat} and \ref{SsecBruhatcosets}, we give a tableau criterion for the Bruhat order on $B_n$ and on cosets, respectively. For this aim, we recall Kashiwara-Nakashima (KN) tableaux in type $C$ and define symplectic key tableaux.

\subsection{Kashiwara-Nakashima tableaux in type $C$}
This subsection focuses on the notion of symplectic tableaux introduced by
Kashiwara and Nakashima to label the vertices of the type $C$ crystal graphs \cite{NK 94}, which are a variation of the De Concini tableaux \cite{DeC 79}. (See \cite{She 99} for more details.)

A vector $\lambda=(\lambda_1,\dots, \lambda_n)\in\Z^n$ is a partition of $|\lambda|=\sum\limits_{i=1}^n \lambda_i$ if $\lambda_1\geq \lambda_2\geq\dots\geq\lambda_n\geq 0$.
The \emph{Young diagram} of shape $\lambda$ is an array of boxes, left justified, in which the $i$-th row, from top to bottom, has $\lambda_i$ boxes. We identify a partition with its Young diagram.
For example, the Young diagram of shape $\lambda=(2,2,1)$ is $\YT{0.17in}{}{
	{{},{}},
	{{},{}},
	{{}}}$.

Given $\mu$ and $\nu$ two partitions with $\nu\leq \mu$ entrywise, we write $\nu\subseteq \mu$. The Young diagram of shape $\mu/\nu$ is obtained after removing the boxes of the Young diagram of $\nu$ from the Young diagram of $\mu$.
For example, the Young diagram of shape $\mu/\nu=(2,2,1)/(1,0,0)$ is $\begin{tikzpicture}[scale=.4, baseline={([yshift=-.8ex]current bounding box.center)}]
\draw (1,0) rectangle +(-1,-1);
\draw (1,-1) rectangle +(-1,-1);
\draw (0,-2) rectangle +(-1,-1);
\draw (0,-1) rectangle +(-1,-1);
\end{tikzpicture}$\,.
\begin{definition}
	Let $\nu\subseteq \mu$ be two partitions and $A$ a completely ordered alphabet. A \emph{semistandard skew tableau} of shape $\mu/\nu$ on the alphabet $A$ is a filling of the diagram $\mu/\nu$ with letters from $A$, such that the entries are strictly increasing in each column and weakly increasing in each row. When $|\nu|=0$ then we obtain a semistandard Young tableau of shape $\mu$.
\end{definition}

Denote by $SSYT(\mu/\nu, A)$
the set of all semistandard Young skew tableaux $T$ of shape $\mu/\nu$, with entries in $A$. When $A=[n]$ we write $SSYT(\mu/\nu, n)$.

When considering tableaux with entries in $[\pm n]$, it is usual to have some extra conditions besides being semistandard. We will use a family of tableaux known as \emph{Kashiwara-Nakashima} tableaux.
From now on we consider tableaux on the alphabet $[\pm n]$.

A \emph{column} is a strictly increasing sequence of numbers in $[\pm n]$ and it is usually displayed vertically.
A column is said to be \emph{admissible} if the following \emph{one column condition} (1CC) holds for that column:

\begin{definition}[1CC]\label{1CC}
	Let $C$ be a column. The $1CC$ holds for $C$ if for all pairs $i$ and $\overline{i}$ in $C$, where $i$ is in the $a$-th row counting from the top of the column, and $\overline{i}$ in the $b$-th row counting from the bottom, we have $a+b\leq i$.
\end{definition}
If a column $C$ satisfies the $1CC$ then $C$ has at most $n$ letters.

If $1CC$ doesn't hold for $C$ we say that $C$ \emph{breaks the $1CC$ at $z$}, where $z$ is the minimal positive integer such that $z$ and $\overline{z}$ exist in $C$ and there are more than $z$ numbers in $C$ with absolute value less or equal than $z$. 

\begin{example}
	The column $\YT{0.17in}{}{
		{{1}},
		{{2}},
		{{\overline{1}}}}$ breaks the $1CC$ at $1$.
\end{example}


The following definition states conditions to when $C$ can be \emph{split}:
\begin{definition}
	Let $C$ be a column and let $I = \{z_1 > \dots > z_r\}$ be the
	set of unbarred letters $z$ such that the pair $(z, \overline{z})$ occurs in $C$. The column
	$C$ can be split when there exists a set of $r$ unbarred
	letters $J = \{t_1 > \dots > t_r\} \subseteq [n]$ such that:
	\begin{enumerate}
		\item $t_1$ is the greatest letter of $[n]$ satisfying $t_1 < z_1$,  $t_1 \not\in C$, and $\overline{t_1}\not\in C$,
		\item for $i=2, \dots, r$, we have that  $t_i$ is the greatest letter of $[n]$ satisfying $t_i <\! \min(t_{i-1},   z_i)$,  $t_i \not\in C$, and $\overline{t_i} \not\in C$.
	\end{enumerate}\end{definition}
	
	The $1CC$ holds for a column $C$ if and only if $C$ can be split  \cite[Lemma 3.1]{She 99}.
	If $C$ can be split then we define right column of $C$, $rC$, and the left column of $C$, $\ell C$, as follows:
	
	\begin{enumerate}
		\item $rC$ is the column obtained by changing in $C$,  $\overline{z_i}$ into $\overline{t_i}$ for each letter $z_i \in I $ and by reordering if necessary,	
		\item $\ell C$ is the column obtained by changing in $C$, $z_i$ into $t_i$ for each letter $z_i \in I $ and by reordering if necessary.
	\end{enumerate}
	
	If $C$ is admissible then $\ell C\leq C \leq rC$ by entrywise comparison. If $C$ doesn't have symmetric entries, then $C$ is admissible and  $\ell C= C =rC$.
	In the next definition we give conditions for a column $C$ to be \emph{coadmissible}.
	\begin{definition}
		We say that a column $C$ is coadmissible if for every pair $i$ and $\overline{i}$ on $C$, where $i$ is on the $a$-th row counting from the top of the column, and $\overline{i}$ on the $b$-th row counting from the top, then $b-a\leq n-i$.
	\end{definition}
	Note that, unlike in Definition \ref{1CC}, in the last definition $b$ is counted from the top of the column.

	Given an admissible column $C$, consider the function $\Phi$ that sends $C$ to the column of the same size in which the unbarred entries are taken from $\ell C$ and the barred entries are taken from $rC$. The column $\Phi(C)$ is a coadmissible column and the algorithm to form $\Phi(C)$ from $C$ is reversible \cite[Section 2.2]{Lec 02}. In particular, every column on the alphabet $[n]$ is simultaneously admissible and coadmissible. 
	
	\begin{example} Let
		$C=\YT{0.17in}{}{
			{{2}},
			{{3}},
			{{\overline{3}}}}$ be an admissible column. Then $\ell C=\YT{0.17in}{}{
			{{1}},
			{{2}},
			{{\overline{3}}}}$ and $rC=\YT{0.17in}{}{
			{{2}},
			{{3}},
			{{\overline{1}}}}$. So $\Phi(C)=\YT{0.17in}{}{
			{{1}},
			{{2}},
			{{\overline{1}}}}$ is coadmissible.
	\end{example}
	
	Let $T$ be a skew tableau with all of its columns admissible. The split form of a skew tableau $T$, $spl(T)$, is the skew tableau obtained after replacing each column $C$ of $T$ by the two columns $\ell C\,rC$. The tableau $spl(T)$ has double the amount of columns of $T$.
	
	\begin{definition}
		A semistandard skew tableau $T$ is a \emph{Kashiwara-Nakashima (KN) skew tableau} if its split form is a semistandard skew tableau. We define $\mathcal{KN}(\mu/\nu, n)$ to be the set of all KN tableaux of shape $\mu/\nu$ in the alphabet $[\pm n]$. When $\nu=0$ we obtain $\mathcal{KN}(\mu, n)$. 
	\end{definition}

	\begin{example} The split of the tableau 
		$T=\YT{0.17in}{}{
			{{2},{2}},
			{{3},{3}},
			{{\overline{3}}}}$ is the tableau $spl(T)=\YT{0.17in}{}{
			{{1}, {2},{2},{2}},
			{{2},{3},{3},{3}},
			{{\overline{3}},{{\overline{1}}}}}$. Hence $T\in \mathcal{KN}((2,2,1),3)$.
	\end{example}
	
	If $T$ is a tableau without symmetric entries in any of its columns, i.e., for all $i\in[n]$ and for all columns $C$ in $T$, $i$ and $\overline{i}$ do not appear simultaneously in the entries of $C$, then in order to check whether $T$ is a KN tableau it is enough to check whether $T$ is semistandard in the alphabet $\left[\pm n\right]$. In particular $SSYT(\mu/\nu, n)\subseteq \mathcal{KN}(\mu/\nu, n)$.
	
	The \emph{weight} of a word $w$ on the alphabet $[\pm n]$ is defined to be the vector $\text{wt}(w)\in \Z^n$ where the entry $i$ is obtained by adding the multiplicity of the letter $i$ and subtracting the multiplicity of the letter $\overline{i}$, for $i\in [n]$.
	If $T$ is a skew tableau, the \emph{column reading} of $T$, $cr(T)$, is
	the word read in $T$ in the Japanese way, column reading top to bottom and right to left.
	The \emph{length} of $w$ is the total number of letters in $w$. The weight of a KN tableau $T$ is the vector $\text{wt}\, T:=\text{wt}(cr(T))=(t_1-t_{\overline{1}}, t_2-t_{\overline{2}},\dots, t_n-t_{\overline{n}})\in \Z^n$, where $t_\alpha$ is the number of $\alpha$'s in $T$, with $\alpha \in [\pm n]$.
	
	\begin{example} Let
		$T=\YT{0.17in}{}{
			{{2},{2}},
			{{3},{3}},
			{{\overline{3}}}}$ and $n=3$. Then $cr(T)=23\,23\overline{3}$ and $\text{\emph{wt}}(T)=\text{\emph{wt}}(cr(T))=(0,2,1)$.\end{example}
	
	In Section \ref{BakerIns}, we recall a way to go from a word in the alphabet $[\pm n]$ to a KN tableau, the Baker-Lecouvey insertion.
	
	\subsection{Key tableaux in type $C$ and the Bruhat order on $B_n$}\label{keyBruhat}
	\begin{definition}
		A \emph{key tableau} in type $C_n$ is a KN tableau in $\mathcal{KN}(\lambda, n)$, in which the set of elements of each column is contained in the set of elements of the previous column and the letters $i$ and $\overline{i}$ do not appear simultaneously as entries, for any $i\in [n]$.
	\end{definition}
	\begin{example}
		The KN tableau $T=\YT{0.17in}{}{
			{{2},{2}},
			{{3},{\overline{1}}},
			{{\overline{1}}}}$ is a key tableau.
	\end{example}
	
	The set of key tableaux in type $A$ is the subset of the key tableaux in type $C$ consisting of the tableaux having only positive entries, hence they are SSYT for the alphabet $[n]$.
	
	Every vector $v$ of  $\Z^n$ is in the $B_n$-orbit of exactly one partition,  $\lambda_v$, which is the one obtained  by sorting the absolute values of all entries of $v$. Given a partition $\lambda\in\Z^n$, the \emph{$B_n$-orbit of $\lambda$} is the set $B_n\lambda :=\{\sigma\lambda\mid \sigma\in B_n\}$.
	For instance,	the vector $v=(1,\overline{3},0, 3, \overline{2})$ is in the $B_5$-orbit of $\lambda=(3,3,2,1,0)$.

	
	
	\begin{proposition}\label{UniKw}
		Let $\lambda$ be a partition and $v\in B_n\lambda $. There is exactly one key tableau $K(v)$ whose weight is $v$. In addition, the shape of the key tableau $K(v)$ is $\lambda$. When $v=\lambda$, $K(\lambda)$ is the only KN tableau of weight and shape $\lambda$, also called Yamanouchi tableau of shape $\lambda$.
	\end{proposition}
	\begin{proof} \underline{Existence:} Given $v=(v_1, \dots, v_n)\in \Z^n$ there exists a key tableau $K$ of weight $v$ by putting in the first $\left|v_i\right|$ columns the letter $i$ if $v_i\geq 0$ or $\overline{i}$ if $v_i\leq 0$, and then sorting the columns properly. Clearly the columns of $K$ are nested and it is a KN tableau without symmetric entries, hence it is a key tableau. Also, its shape is $\lambda_v=\lambda$.

		\underline{Uniqueness:}
		Since the key tableaux don't have symmetric entries then, for all $i\in[n]$, we have that in $K$ the letter $sgn(v_i)i$ appears $\left|v_i\right|$ times in its entries. In order to the columns of $K$ be nested we have that these $\left|v_i\right|$ entries appear in the first  $\left|v_i\right|$ columns, hence we have determined exactly which letters appear in each column of $K$ and now we just have to order them correctly. So the key tableau $K$ with weight $v$ is unique. When $v=\lambda$, $K(\lambda)$ has only $i$'s in the row $i$, for $i \in [n]$.
	\end{proof}
	
	\begin{example}
		Let $v=(1,\overline{3},0, 3, \overline{2})$. Then $K(v)=~\YT{0.2in}{}{
			{{1},{4},{4}},
			{{4},{\overline{5}},{\overline{2}}},
			{{\overline{5}},{\overline{2}}},
			{{\overline{2}}}}$.
	\end{example}
	
	Hence there is a bijection between vectors in  $B_n\lambda $ and the key tableaux in type $C$ on the alphabet $[\pm n]$ with shape $\lambda$, given by the map $v\mapsto K(v)$. If $\sigma\in B_n$ we put $K(\sigma):=K(\sigma\Delta^n)$, where $\Delta^n=(n,n-1,\dots, 1)$ is the staircase partition in $\Z^n$. One has a natural bijection between $B_n$ and the $B_n$-orbit of $\Delta^n$.
	
	\begin{proposition}\label{actionv}
		If $\sigma \in B_n$ has the letter $\alpha$ in the $j$-th position then $\alpha$ appears in the first $n+1-j$ columns of the corresponding key tableau, $K(\sigma)$.
	\end{proposition}
	\begin{proof}
		Put $\Delta:=\Delta^n$. Remember that, ignoring signs, $\sigma\Delta  =(\Delta_{\sigma^{-1}(1)}, \dots, \Delta_{\sigma^{-1}(n)})$, with $\Delta=(n, \dots, 1)$. The $i$-th letter of $\sigma\Delta $ has negative sign if and only if $\overline{i}$ appears in $\sigma$.
		If $\alpha$ is positive, then in the position $\alpha$ of $\sigma\Delta $ will appear $\Delta_j=n+1-j$.
		If $\alpha$ is negative, then in the position $-\alpha$ will appear $\overline{\Delta_j}=\overline{n+1-j}$.
	\end{proof}
	
	We now append $0$ to the alphabet $\left[\pm n\right]$, obtaining $\left[\pm n\right]\cup \{0\}$, where $n<0<\overline{n}$, and, for all $\sigma\in B_n$, we put $\sigma(0):=0$. Given an element $\sigma \in B_n$ consider the map 
	\begin{align*}
		\left[\pm n\right]\cup \{0\}\times\left[\pm n\right]\cup \{0\}&\rightarrow \mathbb{N}_0\\
		(i,j)&\mapsto \left|\{a \leq i:  \sigma(a)\geq j\}\right|:=\sigma[i,j].
	\end{align*}
	
	This map, originally defined in \cite{BB05}, produces a table which is related to key tableaux in type $C$. See example below:
	\begin{example}
		Let $\sigma =[\overline{3}\,\overline{1}\,2\,4]$. Then $\sigma(4,3,2,1)=(\overline{3},2,\overline{4},1)$ and
		
		$K(\sigma)=~\YT{0.2in}{}{
			{{2},{2},{\overline{3}},{\overline{3}}},
			{{4},{\overline{3}},{\overline{1}}},
			{{\overline{3}},{\overline{1}}},
			{{\overline{1}}},}$

		The family of numbers $\sigma [i,j]$ where $(i,j)\in [\pm n]\cup \{0\}\times [\pm n]\cup \{0\}$ originates the following table, where $i$ indexes the columns, left to right, and $j$ indexes the rows, top to bottom. We add a row of zeros at the bottom for convenience:
		
		\begin{table}[h]
			\begin{tabular}{llllllllll}
				& 1                      & 2                      & 3                      & 4   &0	& $\overline{4}$                      & $\overline{3}$                      &$\overline{2}$                      &$\overline{1}$                  \\ \cline{2-10} 
				\multicolumn{1}{l|}{1} & \multicolumn{1}{l|}{1} & \multicolumn{1}{l|}{2} & \multicolumn{1}{l|}{3} & \multicolumn{1}{l|}{4}& \multicolumn{1}{l|}{5} & \multicolumn{1}{l|}{6} & \multicolumn{1}{l|}{7} & \multicolumn{1}{l|}{8}& \multicolumn{1}{l|}{9} \\ \cline{2-10} 
				\multicolumn{1}{l|}{2} & \multicolumn{1}{l|}{1} & \multicolumn{1}{l|}{2} & \multicolumn{1}{l|}{3} & \multicolumn{1}{l|}{4}& \multicolumn{1}{l|}{5} & \multicolumn{1}{l|}{6} & \multicolumn{1}{l|}{7} & \multicolumn{1}{l|}{7}& \multicolumn{1}{l|}{8} \\ \cline{2-10}
				\multicolumn{1}{l|}{3} & \multicolumn{1}{l|}{1} & \multicolumn{1}{l|}{2} & \multicolumn{1}{l|}{2} & \multicolumn{1}{l|}{3}& \multicolumn{1}{l|}{4} & \multicolumn{1}{l|}{5} & \multicolumn{1}{l|}{6} & \multicolumn{1}{l|}{6}& \multicolumn{1}{l|}{7} \\ \cline{2-10} 
				\multicolumn{1}{l|}{4} & \multicolumn{1}{l|}{1} & \multicolumn{1}{l|}{2} & \multicolumn{1}{l|}{2} & \multicolumn{1}{l|}{3}& \multicolumn{1}{l|}{4} & \multicolumn{1}{l|}{5} & \multicolumn{1}{l|}{6} & \multicolumn{1}{l|}{6}& \multicolumn{1}{l|}{6} \\ \cline{2-10}
				\multicolumn{1}{l|}{0} & \multicolumn{1}{l|}{1} & \multicolumn{1}{l|}{2} & \multicolumn{1}{l|}{2} & \multicolumn{1}{l|}{2}& \multicolumn{1}{l|}{3} & \multicolumn{1}{l|}{4} & \multicolumn{1}{l|}{5} & \multicolumn{1}{l|}{5}& \multicolumn{1}{l|}{5} \\ \cline{2-10}
				\multicolumn{1}{l|}{$\overline{4}$} & \multicolumn{1}{l|}{1} & \multicolumn{1}{l|}{2} & \multicolumn{1}{l|}{2} & \multicolumn{1}{l|}{2}& \multicolumn{1}{l|}{2} & \multicolumn{1}{l|}{3} & \multicolumn{1}{l|}{4} & \multicolumn{1}{l|}{4}& \multicolumn{1}{l|}{4} \\ \cline{2-10}
				\multicolumn{1}{l|}{$\overline{3}$} & \multicolumn{1}{l|}{1} & \multicolumn{1}{l|}{2} & \multicolumn{1}{l|}{2} & \multicolumn{1}{l|}{2}& \multicolumn{1}{l|}{2} & \multicolumn{1}{l|}{2} & \multicolumn{1}{l|}{3} & \multicolumn{1}{l|}{3}& \multicolumn{1}{l|}{3} \\ \cline{2-10}
				\multicolumn{1}{l|}{$\overline{2}$} & \multicolumn{1}{l|}{0} & \multicolumn{1}{l|}{1} & \multicolumn{1}{l|}{1} & \multicolumn{1}{l|}{1}& \multicolumn{1}{l|}{1} & \multicolumn{1}{l|}{1} & \multicolumn{1}{l|}{2} & \multicolumn{1}{l|}{2}& \multicolumn{1}{l|}{2} \\ \cline{2-10}
				\multicolumn{1}{l|}{$\overline{1}$} & \multicolumn{1}{l|}{0} & \multicolumn{1}{l|}{1} & \multicolumn{1}{l|}{1} & \multicolumn{1}{l|}{1}& \multicolumn{1}{l|}{1} & \multicolumn{1}{l|}{1} & \multicolumn{1}{l|}{1} & \multicolumn{1}{l|}{1}& \multicolumn{1}{l|}{1} \\ \cline{2-10}
				\multicolumn{1}{l|}{} & \multicolumn{1}{l|}{0} & \multicolumn{1}{l|}{0} & \multicolumn{1}{l|}{0} & \multicolumn{1}{l|}{0}& \multicolumn{1}{l|}{0} & \multicolumn{1}{l|}{0} & \multicolumn{1}{l|}{0} & \multicolumn{1}{l|}{0}& \multicolumn{1}{l|}{0} \\ \cline{2-10}
			\end{tabular}
		\end{table}
		
		
		To go from the table to the key tableau note that, for $i\in [n]$, the $i$-th column of the table encodes the  $(n+1-i)$-th column of the tableau, in the sense that if we look to the the $i$-th column of the table, from bottom to top, if the entry of the table increases in one unity then the index of the row associated to that entry exists in the $(n+1-i)$-th column of the tableau. Knowing the entries in a column of a tableau, its ordering is unique. 
		The columns of the tableau constructed this way are nested because the indexes in which the column $i$ increases are $\sigma(j)$, for $j\leq i$. So the tableau taken from the table is the key tableau $K(\sigma)$. It is also possible to construct the table from the key tableau and then we only need the first $n$ columns of the table. 
		
	\end{example}
	
	
	
	

	We then have the following result:
	\begin{proposition}
		Consider $\sigma, \rho\in B_n$. $K(\sigma )\ge K(\rho)$ entrywise if and only if $\sigma [i,j]\ge \rho[i,j]$, where $i \in [n], $ and $j\in [\pm n]$.
	\end{proposition}
	
	In \cite[Theorem 8.1.8]{BB05} it is proved that, for $\sigma, \rho \in B_n$, $\sigma \leq \rho$ in the Bruhat order if and only if $\sigma[i,j]\leq \rho[i,j]$ for all $i,j\in[\pm n]$. But the result in \cite[Theorem 8.1.7]{BB05} implies that we only need to compare $\sigma [i,j]$ and $\rho[i,j]$ for $i\in [n]$. Henceforth, we have the following criterion for the Bruhat order on $B_n$:
	
	\begin{theorem}\label{BruhatKey}
		Consider $\sigma, \rho\in B_n$. $K(\sigma )\ge K(\rho)$ entrywise if and only if $\sigma \ge \rho$ in the Bruhat order.
	\end{theorem}
	
	\begin{remark}
		In \cite[Chapter 8.1]{BB05} the authors use the same alphabet as here, but with the usual ordering on the integers. So, to translate the results from there to here, it is needed to apply the ordering isomorphism defined by: $i\mapsto \overline{n-i+1}$ if $i\in [n]$; $i\mapsto n+i+1$ if $i\in -[n]$; $0\mapsto 0$. Using the usual ordering, the authors give a tableau criterion for the Bruhat order in Exercise $6$, pp. $287$--$288$, which is effectively the transpose version of the tableau criterion presented here.	
		Also note that the generators used in \cite[Chapter 8.1]{BB05} are the same used here, although with different indexation. Our generator $s_i$ corresponds to the generator $s_{n-i}$ in \cite[Chapter 8.1]{BB05}, for all $i \in [n]$.
	\end{remark}
	
	\subsection{The Bruhat order on cosets of $B_n$}\label{SsecBruhatcosets}

	Consider a partition $\lambda\in \Z^n$.
	Let $W_\lambda=\{\rho\in B_n\mid \rho\lambda=\lambda\}$ be the stabilizer of $\lambda$, under the action of $B_n$. Since $\lambda$ is a partition, $W_\lambda$ is a subgroup of $B_n$ generated by some of the generators of $B_n$. Let $J\subseteq [n]$ be the set of the indices of the generators of $W_\lambda$, i.e. $W_\lambda=\langle s_j, j\in J \rangle$, and $J^c$ the complement of this set in $[n]$.
	Let $ B_n/W_\lambda=\{\sigma W_\lambda:\sigma\in B_n\}$ be the set of left cosets of $B_n$ determined by the subgroup $W_\lambda$. Each coset $\sigma W_\lambda$ returns a unique vector $v$ when acting on $\lambda$, and has a unique minimal length element $\sigma_v$, such that $v=\sigma_v\lambda$. Reciprocally, given a vector $v\in B_n\lambda$,  there is a unique minimal length element $\sigma_v\in B_n$ such that $v=\sigma_v\lambda$. We have then a bijection between the vectors in $B_n\lambda $  and the left cosets of $B_n$, determined by the subgroup $W_\lambda$, given by 
	$v\mapsto \sigma_v W_\lambda $.
	The set $J^c$ detects the minimal length coset representatives of $B_n/W_\lambda$: $\sigma$ is a minimal coset representative of $B_n/W_\lambda$ if and only if all its reduced decompositions end with a generator $s_i\in J^c$ \cite{BB05}. However key tableaux, $K(v)$, $v\in B_n\lambda $, may be used
	to explicitly construct the minimal length coset representatives of  $B_n/W_\lambda$.
	Given a vector $v\in  B_n\lambda$, we show that there is a unique minimal length element $\sigma_v\in B_n$ such that $v=\sigma_v\lambda$ and we show how to obtain $\sigma_v$ explicitly. The next proposition is a generalization of what Lascoux does in \cite{Las 03} for vectors in $\N^n$ (hence $\sigma_v \in \mathfrak{S}_n$).

	
	\begin{proposition}\label{minimalBn}
		Let $v\in B_n\lambda $ and $T$ the tableau obtained after adding the column $C=\YT{0.17in}{}{
			{{1}},
			{{2}},
			{{\vdots}},
			{{n}}}$ to the left of $K(v)$. The minimal length element $\sigma\in B_n$, modulo $W_\lambda$, is given by the reading word of $T$ where entries with the same absolute value are read just once.
	\end{proposition}
	\begin{proof}
		Consider $\lambda=(\lambda_1, \dots, \lambda_n)$. Let $a_i$ be the multiplicity of $i$ in $\lambda$, for $0\leq i\leq \lambda_1$. In this proof we will write $\lambda$ as $(\lambda_1^{a_{\lambda_1}},$ $(\lambda_1-1)^{a_{\lambda_1-1}},\dots ,1^{a_1},0^{a_0})$. 
		Note that $\sum\limits_{i=0}^{\lambda_1}a_i=n$.
		
		Let $\sigma=[\alpha_1\dots\alpha_n]\in B_n$ read from $T$. Let's prove that $\alpha_j$ appears $\lambda_j$ times in $K(v)$:
		If $j=1$ then $\alpha_1$ appears in all columns of $K(v)$, because it was the first letter read and the columns are nested. Hence it appears $\lambda_1$ times. Also, the $\left|\alpha_1\right|$-th entry of $\lambda\sigma$ is $sgn(\alpha_1)\lambda_1$ which is the weight of $\left|\alpha_1\right|$ in $K(v)$.
		For $j>1$, proceeding inductively, we have that $\alpha_j$ appears in all columns of $K(v)$ not fully occupied by $\alpha_i$, with $i<j$, hence it appears $\lambda_j$ times. Also, the $\left|\alpha_j\right|$-th entry of $\lambda\sigma$ is $sgn(\alpha_j)\lambda_j$, which is the weight of $\left|\alpha_j\right|$ in $K(v)$. This makes sense even if $\lambda_j=0$.
		So we have that $\sigma\lambda=v$.
		
		We only have to see that $\sigma$ is the minimal length element of the set $\{\rho\in B_n\mid \rho\lambda=v\}$.
		The subset of elements $B_n$ that applied to $\lambda$ returns $v$ is the coset $\sigma W_\lambda $. Looking at $\sigma$, this allows us to swap $\alpha_i$ and $\alpha_j$ in $\sigma$ if $\lambda_i=\lambda_j$ and to change the sign of $\alpha_i$ if $\lambda_i=0$. Since for each column the reading to obtain $\sigma$ is ordered from the least to the biggest, we have that between these elements of $B_n$, $\sigma$ has minimal number of inversions and the letter $\alpha_j$ is unbarred if $\lambda_j=0$ because $\alpha_j$ will only be added to $\sigma$ when reading the column $C$. Hence, by Proposition \ref{length}, 	$\sigma$ is the minimal length element of $\sigma W_\lambda $.
	\end{proof}
	Given a partition $\lambda\in \Z^n$ we identify each coset $\sigma W_\lambda$ with its minimal length representative $\sigma_v$, where $v=\sigma\lambda\in B_n\lambda $. Under this identification, we now induce the Bruhat order in the $B_n$-orbit of $\lambda$ and in the coset space of $ B_n/W_\lambda$.
	\begin{definition}\label{BruhatVec}
		Consider the vectors $v, w\in B_n\lambda$, where $\lambda$ is a partition. We say that $v\leq w$, in the Bruhat order, if $\sigma_v\leq \sigma_w$.
	\end{definition}
	Let $v\in B_n\lambda $. If $K:=K(v)$ is the key tableau with weight $v$, consider the tableau $\widetilde{K}$ obtained from $K$ after erasing the minimal number of columns in order to have a tableau with no duplicated columns. Let $\widetilde{v}$ and $\widetilde{\lambda}$ be the weight and  the shape of $\widetilde{K}$, respectively. If $K$ and $K'$ are two key tableaux with shape $\lambda$, we have that $K\geq K'$ (by entrywise comparison) if and only if $\widetilde{K}\geq \widetilde{K'}$. Note that to recover $K$ from $\widetilde{K}$ we just have to know $\lambda$, and that $\widetilde{K}=K(\widetilde{v})$.
	
	It is also possible to obtain $\widetilde{v}$ from $v$ without having to look to key tableaux. If $i$ is positive, $i$ and $\overline{i}$ do not appear in $v$ but $i+1$ or $\overline{i+1}$ appears then change all appearances of $i+1$ and $\overline{i+1}$ to $i$ and $\overline{i}$, respectively, and repeat this as many times as possible, obtaining the vector $\widetilde{v}$. The set of the absolute values of its entries is a set of consecutive integers starting either in  $0$ or $1$. Hence the key tableau associated to it doesn't have repeated columns.
	
	Due to Proposition \ref{minimalBn} we have that $\sigma_{\tilde{v}}=\sigma_v$ and  $\widetilde{v}=\sigma_{\tilde{v}}\widetilde{\lambda_v}=\widetilde{\sigma_v\lambda_v}$.
	\begin{example}
		Consider $v=(1,0,\overline{3},3,\overline{5})\in B_5(5,3,3,1,0)$. Hence $K(v)=\YT{0.17in}{}{
			{{1},{4},{4},{\overline{5}},{\overline{5}}},
			{{4},{\overline{5}},{\overline{5}}},
			{{\overline{5}},{\overline{3}},{\overline{3}}},
			{{\overline{3}}}}$ has shape $\lambda=(5,3,3,1,0)$, weight $v$ and $\sigma_v=[\overline{5}\,4\,\overline{3}\,1\,2]$.
		Now note that $\widetilde{v}=(1,0,\overline{2},2,\overline{3})$, hence $K(\widetilde{v})=\YT{0.17in}{}{
			{{1},{4},{\overline{5}}},
			{{4},{\overline{5}}},
			{{\overline{5}},{\overline{3}}},
			{{\overline{3}}}}=\widetilde{K(v)}$ has shape $(3,2,2,1,0)=\widetilde{\lambda}$ and $\sigma_{\tilde{v}}=[\overline{5}\,4\,\overline{3}\,1\,2]=\sigma_v$.
	\end{example}
	
	Recall $J$ and $J^c$ defined above. Note that the set $J$ is the same for $\lambda$ and $\widetilde{\lambda}$. If $i\in J^c$ and $i=n$ then all entries of $\lambda$ are different from $0$, which implies $K(v)$ (and $\widetilde{K(v)}$) having columns of length $n$; if $i\in J^c$ and $i<n$ then $\lambda_i> \lambda_{i+1}$, hence $K(v)$ will have exactly $i$ rows with length greater then $\lambda_{i+1}$, hence $K(v)$ (and $\widetilde{K(v)}$) will have columns of length $i$.
	Since $\widetilde{K(v)}$ doesn't have repeated columns, $J^c$ have exactly the information of what column lengths exist in $\widetilde{K(v)}$.
	Theorem $3BC$ of Proctor's Ph.D. thesis \cite{Pro 81} states that given a partition $\lambda$ there is a poset isomorphism between the poset formed by the key tableaux of shape $\widetilde{\lambda}$ (ordered by entrywise comparison) and the poset formed by the Bruhat order in the vectors of the orbit $B_n\tilde{\lambda} =\{\sigma\tilde{\lambda}:\sigma\in B_n \}$.
	
	The following theorem gives a tableau criterion for the Bruhat order on vectors in the same $B_n$-orbit and for the corresponding $B_n$-coset space.
	
	\begin{theorem}\label{keycoset}
		Let $v,\,u\in B_n\lambda $. Then $\sigma_v\leq\sigma_u$ if and only if $K(v)\leq K(u)$.
	\end{theorem}
	\begin{proof} We have that 
		\begin{align*}
			\sigma_v\leq\sigma_u \stackrel{(1)}{\Leftrightarrow} v \leq u\stackrel{(2)}{\Leftrightarrow} \widetilde{v} \leq \widetilde{u} \stackrel{(3)}{\Leftrightarrow} K(\widetilde{v}) \leq K(\widetilde{u})\Leftrightarrow \widetilde{K(v)} \leq \widetilde{K(u)}\stackrel{(4)}{\Leftrightarrow} K(v) \leq K(u),
		\end{align*}
		where $(1)$ holds by Definition \ref{BruhatVec}. Note that in $(2)$ we also need to record $\lambda$, because it is needed in $(4)$ to recover the shape of $K(v)$ from the shape $\widetilde{K(v)}$. Finally the equivalence $(3)$ is an application of Theorem 3BC of Proctor's Ph.D. thesis \cite{Pro 81}.
	\end{proof}
	
	The following example illustrates Theorem \ref{keycoset}.
	
	\begin{example} Here we have two vectors with the respective key tableaux, ordered by entrywise comparison. The corresponding minimal coset representatives, calculated using Proposition \ref{minimalBn}, preserve this order.

		$K(3,\overline{3},0, 0, \overline{2})=\YT{0.15in}{}{
			{{1},{1},{1}},
			{{\overline{5}},{\overline{5}},{\overline{2}}},
			{{\overline{2}},{\overline{2}}}}$ $
		\le$ $ K(\overline{3},2,0,\overline{3},0)=\YT{0.15in}{}{
			{{2},{2},{\overline{4}}},
			{{\overline{4}},{\overline{4}},{\overline{1}}},
			{{\overline{1}},{\overline{1}}}}
		$ and $\sigma_v=[1\overline{2}\overline{5}34]\le \sigma_u=[\overline{4}\overline{1}235]$.
	\end{example}
	
	\section{Crystal graphs in type $C$ and symplectic plactic monoid}\label{SecRS}
	We recall two equivalence relations of words in the alphabet $\left[\pm n\right]$, the type $C$ Knuth equivalence, or (symplectic) plactic equivalence, and the (symplectic) coplactic equivalence. On the basis of these two equivalence relations is the Robinson-Schensted type $C$ correspondence, in which each word is uniquely parametrized by a KN tableau and an oscillating tableau of the same final shape. This bijection has a natural interpretation in terms of crystal connectivity and crystal isomorphic connected components in Kashiwara theory of crystal graphs \cite{BSch 17, Kash 95, Lec 02, Lec 07}. 
	For this aim and reader convenience, we begin to recall the Sheats symplectic \emph{jeu de de taquin} and Baker-Lecouvey insertion.
	\subsection{Sheats symplectic \textit{jeu de taquin}}
	The symplectic \textit{jeu de taquin} \cite{Lec 02, She 99} is a procedure that allows us to change the shape of a  KN skew tableau and eventually rectify it.  
	
	To explain how the symplectic \textit{jeu de taquin} behaves, we need to look to how it works on $2$-column KN skew tableaux. Let $T$ be a $2$-column KN skew tableau with splittable columns $C_1$ and $C_2$ such that $C_1$ has an empty cell. 
	
	Consider the tableau $spl(T)$ such that the columns $\ell C_1$ and $rC_1$ have an empty cell in the same row as $C_1$. Let $\alpha$ be the entry under the empty cell of $rC_1$ and $\beta$ to the entry right of the empty cell of $rC_1$.
	
	If $\alpha\leq \beta$ or $\beta$ does not exist,  then the empty cell of $T$ will change its position with the cell beneath it. This is a vertical slide.

	If the slide is not vertical, then it is horizontal. So we have $\alpha> \beta$ or $\alpha$ does not exist. Let $C_1'$ and $C_2'$ be the columns after the slide. In this case we have two subcases, depending on the sign of $\beta$:
	\begin{enumerate}
		\item If $\beta$ is barred we are moving a barred letter from $\ell C_2$ to $rC_1$. Remember that $\ell C_2$ has the same barred part as $C_2$ and that $rC_1$ has the same barred part as $\Phi(C_1)$. So, looking at $T$, we will have an horizontal slide of the empty cell, $C_2'=C_2\setminus \{\beta\}$ and $C_1'=\Phi^{-1}(\Phi(C_1)\cup \{\beta\})$. In a sense, $\beta$ went from $C_2$ to $\Phi(C_1)$.
		
		\item If $\beta$ is unbarred we have a similar story, but this time $\beta$ will go from $\Phi(C_2)$ to $C_1$, hence $C_1'=C_1\cup \{\beta\}$ and $C_2'=\Phi^{-1}(\Phi(C_2)\setminus \{\beta\})$. Although in this case it may happen that $C_1'$ is no longer admissible. In this case, if the 1CC breaks at $i$, we erase both $i$ and $\overline{i}$ from the column and remove a cell from the bottom and from the top column, and place all the remaining cells orderly.
		
	\end{enumerate}
	
	Eventually the empty cell will be a cell such that $\alpha$ and $\beta$ do not exist. In this case we redefine the shape to not include this cell and  the \textit{jeu de taquin} ends. A box of the diagram of shape $\lambda$ such that boxes under it and to the right are not in that shape is called an \emph{inner corner}.
	
	Given a KN skew tableau $T$ of shape $\mu/\nu$, the rectification of $T$ consists in playing the \textit{jeu de taquin} until we get a tableau of shape $\lambda$, for some partition $\lambda$. The rectification is a dynamic process, in which the inner shape, and its inner corners, gets redefined after each iteration of the \textit{jeu de taquin}.  The rectification is independent of the order in which the inner corners are filled \cite[Corollary 6.3.9]{Lec 02}. 
	
	\begin{example} Consider the KN skew tableau
		$T=\SYT{0.145in}{}{
			{{2}},
			{{3},{1}},
			{{\overline{1}},{2}}}$.
		We want to rectify it via symplectic \textit{jeu taquin}. We start by splitting and conclude that the first two slides are vertical, obtaining	$\SYT{0.17in}{}{
			{{2}, {2},{1},{1}},
			{{3},{3},{2},{2}},
			{{\overline{1}},{\overline{1}}}}$. Now we do an horizontal slide in which we take $\overline{1}$ from the second column of $T$ and add it to the coadmissible column of the first column of $T$, obtaining the tableau
		$\YT{0.145in}{}{
			{{2},{2}},
			{{3},{3}},
			{{\overline{3}}}}$.
	\end{example}
	
	\begin{remark}
		If the columns $C_1$ and $C_2$ do not have negative entries then the symplectic \textit{jeu de taquin} coincides with the \textit{jeu de taquin} known for SSYT.
	\end{remark}

	\subsection{Baker-Lecouvey insertion}\label{BakerIns}
	
	The Baker-Lecouvey insertion \cite{Ba 00,Lec 02} is a bumping algorithm that given a word in the alphabet $[\pm n]$ returns a KN tableau. Let $w$ be a word in the alphabet $\left[\pm n\right]$, we call $P(w)$ to the tableau obtained after inserting $w$. This insertion is similar to the usual column insertion for SSYT. In fact both have the same behavior unless one the following three cases happens:
	
	Suppose that we are inserting the letter $\alpha$ in the column $C$ of the KN tableau and
	\begin{enumerate}
		\item[(I)] $\overline{y} \in C$ is the smallest letter bigger or equal then $\alpha$ and $y\in C$, for some $y\in [n]$:
		there is in $C$ a maximal string of consecutive decreasing integers $y, y-1, \dots, u+1$ starting in the entry $y$ in the column $C$. Then the bump consists of replacing the entry $\overline{y}$ with $\alpha$ and subtracting $1$ to the entries $y, y-1,\dots,  u+1$. The entry $\overline{u}$ is then inserted in the next column to the right. This is known as the \emph{Type I special bump}.
		
		\item[(II)] if $\alpha=x$ and $\overline{x} \in C$, for some $x\in [n]$: there is a maximal string of consecutive decreasing entries $\overline{x}, \overline{x+1},\dots,\overline{v-1}$ starting in the entry $\overline{x}$ in $C$. Let $\beta$ be the next entry above $\overline{v-1}$. Then we have two subcases: \begin{enumerate}
			\item If $v\leq\beta\leq \overline{v+1}$ then suppose $\delta$ is the smallest entry
			in $C$ which is bigger or equal than $v$. Then this bump consists
			of deleting the entry $\overline{x}$, shifting the entries
			$\overline{x+1}, \dots,  \overline{v-1}$ down one position, inserting $\overline{v}$ where $\overline{v - 1}$ was, and replacing $\delta$
			with $v$. The entry $\delta$ is then bumped into the
			next column. This is known as the \emph{Type IIa special bump}.
			\item If $\beta \leq v-1$ or $\beta$ doesn't exist then there is a maximal string (possibly empty)
			of consecutive integers $v-1, \dots, u+1$
			above the entry $\overline{v-1}$.
			The string is not empty only when $\beta=v-1$, or else the string is empty and $u=v-1$. The bump consists
			of deleting the entry $\overline{x}$, shifting the entries
			$\overline{x + 1}, \dots,  u+1$ down one position,
			and inserting an entry $u$ where $u+1$ (or $\overline{v-1}$, if $\beta \neq v-1$) was. The entry $\overline{u}$ is then bumped into the
			next column. This is known as the \emph{Type IIb special bump}.
		\end{enumerate}
		
		\item[(III)] after adding $\alpha$ in the bottom of the column $C$, the 1CC breaks at $i$: then we will slide out the cells that contain $\overline{i}$ and $i$ via symplectic \textit{jeu de taquin}.
	\end{enumerate}   
	
	In the case III of the Baker-Lecouvey insertion we will be removing a cell from the tableau instead of adding. The length of $cr(P(w))$ might be less than the length of $w$ and the weight is preserved during Baker-Lecouvey insertion, $\text{wt}(w)=\text{wt}(P(w))$. 
	\begin{remark}
		The Baker-Lecouvey insertion is different from what we would have if we use the SSYT column insertion. However, if the word $w$ doesn't have symmetric letters, then the insertion works just like the column insertion for SSYT. Apart from this case, if we were to use SSYT column insertion, the final tableau may not even be a KN tableau. For instance, consider the word $w=2\overline{1}1$.
		The Baker-Lecouvey insertion of $w$ creates the sequence of tableaux
		$\YT{0.17in}{}{
			{{{2}}}}
		\YT{0.17in}{}{
			{{2}},
			{{\overline{1}}}}
		\YT{0.17in}{}{
			{{2},{2}},
			{{\overline{2}}}}=P(2\overline{1}1)$.		
		The SSYT column insertion of $w$ results in the tableau $\YT{0.17in}{}{
			{{1},{2}},
			{{\overline{1}}}}$, which is not a KN tableau because the first column is not admissible.
	\end{remark}
	\begin{example}\label{ExBakIns}
		Consider the word $w=23\overline{2}\overline{3}1$. We now insert all five letters of $w$, obtaining the following sequence of tableaux:	
		$\YT{0.17in}{}{
			{{{2}}}}
		\YT{0.17in}{}{
			{{2}},
			{{3}}}
		\YT{0.17in}{}{
			{{2}},
			{{3}},
			{{\overline{2}}}}
		\YT{0.17in}{}{
			{{1},{\overline{1}}},
			{{3}},
			{{\overline{3}}}}
		\YT{0.17in}{}{
			{{1},{1},{\overline{1}}},
			{{3}},
			{{\overline{3}}}}=P(w).$	
		Note that the insertion of the fourth letter, $\overline{3}$, causes a type I special bump on the first column and  the insertion of the fifth letter, $1$, causes a type IIb special bump on the second column.
	\end{example}

	
	\begin{proposition}\cite[Corollary 6.3.9]{Lec 02}\label{rr}
		Let $T\in \mathcal{KN}(\mu/\nu, n)$. Then the tableau obtained after rectifying $T$ via symplectic \textit{jeu de taquin} coincides with $P(cr(T))$. Moreover, the insertion of $w=w_1\dots w_k$, $P(w)$, is the rectification of the tableau with diagonal shape $\Delta^n/\Delta^{n-1}$ and column reading $w$.
	\end{proposition}
	

	In particular we have that if we insert $cr(T)$ we obtain $T$ again. This implies that during the insertion of $cr(T)$ the case III of the Baker-Lecouvey insertion cannot happen.
	In Example \ref{ExBakIns}, we may conclude that $P(23\overline{2}\overline{3}1)=P(cr(P(23\overline{2}\overline{3}1)))=P(\overline{1}113\overline{3})$.

	\subsection{Robinson-Schensted type $C$ correspondence, plactic and coplactic equivalence}
	
	Let $[\pm n]^\ast$ be the free monoid on the alphabet $[\pm n]$. The Robinson-Schensted type $C$ correspondence \cite[Theorem 5.2.2]{Lec 02} is a combinatorial bijection between words $w\in[\pm n]^\ast$ and pairs $(T, Q)$ where $T$ is a KN tableau and $Q$ is an oscillating tableau, a sequence of Young diagrams that record, by order, the shapes of the tableaux obtained while inserting $w$, whose final shape is the same as $T$. Every two consecutive shapes of the oscillating tableau differ in exactly one cell and its length is the same of $w$.
	Since both the symplectic \textit{jeu de taquin} and the Baker-Lecouvey insertion are reversible \cite{Ba 00, Lec 02}, we have that every pair $(T, Q)$, with the same final shape, is originated by exactly one word.
	The Robinson-Schensted type $C$ correspondence is the following map:
	\begin{align*}
		\left[\pm n\right]^\ast \rightarrow\bigsqcup_\lambda \mathcal{KN}(\lambda, n)\times \mathcal{O}(\lambda, n):\:w\mapsto (P(w), Q(w))
	\end{align*}
	where the union is over all partitions $\lambda$ with at most $n$ parts, and $\mathcal{O}(\lambda, n)$ is the set of all oscillating tableaux with final shape $\lambda$ and all shapes of the sequence have at most $n$ rows.

	\begin{example} In Example \ref{ExBakIns}, the word $w=23\overline{2}\overline{3}1$ is associated to the pair 
		
		$\left(
		\YT{0.17in}{}{
			{{1},{1},{\overline{1}}},
			{{3}},
			{{\overline{3}}}}, 
		\YT{0.17in}{}{
			{{{}}}}
		\YT{0.17in}{}{
			{{}},
			{{}}}
		\YT{0.17in}{}{
			{{}},
			{{}},
			{{}}}
		\YT{0.17in}{}{
			{{},{}},
			{{}},
			{{}}}
		\YT{0.17in}{}{
			{{},{},{}},
			{{}},
			{{}}}\right)$.
	\end{example}
	
	Given $w_1, w_2 \in [\pm n]^\ast$, the relation
	$w_1 \sim w_2\Leftrightarrow P(w_1)=P(w_2)$ defines an equivalence relation on $[\pm n]^\ast$ known as \emph{Knuth equivalence}. The type $C$ plactic monoid is the quotient $[\pm n]^\ast/\sim$ where each Knuth (plactic) class is uniquely identified with a KN tableau \cite{LLT 95, Lec 02}.
	The quotient $[\pm n]^\ast/\sim$ can also be described as the quotient of $[\pm n]^\ast$ by the \emph{elementary Knuth relations}:
	\begin{enumerate}
		\item[K1:] $\gamma \beta \alpha \sim \beta \gamma \alpha$, where $\gamma<\alpha\leq \beta$ and $(\beta, \gamma) \neq (\overline{x}, x)$ for all $x\in[n]$. 
		\item[K2:] $\alpha \beta \gamma \sim \alpha \gamma \beta$, where $\gamma\leq \alpha <\beta$ and  $(\beta, \gamma) \neq (\overline{x}, x)$ for all $x\in[n]$.
		\item[K3:] $y+1 \overline{y+1}\beta\sim \overline{y}y\beta$, where $y<\beta<\overline{y}$ and $y\in [n-1]$.
		\item[K4:] $\beta \overline{y}y\sim\beta y+1 \overline{y+1}$, where $y<\beta<\overline{y}$ and $y\in [n-1]$.
		\item[K5:] $w\sim w\setminus\{z,\overline{z}\}$, where $w\in[\pm n]^\ast$ and $z\in[n]$ are such that $w$ is a non-admissible column that the $1CC$ breaks at $z$, and any proper factor 
		of $w$ is an admissible column.
	\end{enumerate}

	\begin{remark} It can be proved that given a word $w\in [\pm n]^\ast$, any proper factor is admissible if and only if any proper prefix of $w$ is admissible. Thus, in order to be able to apply the Knuth relation $K5$ to a subword $w'$ of $w$, we only need to check that all proper prefixes of $w'$ are admissible, instead of all proper factors.\end{remark}
	
	When Knuth relations are applied to subwords of a word, the weight is preserved while the length may not. Knuth relations can be seen as \textit{jeu de taquin} moves on words or a diagonally shaped tableau, and each symplectic \textit{jeu de taquin} slide preserves the Knuth class of the reading word of a tableau \cite[Theorem 6.3.8]{Lec 02}.
	In Example \ref{ExBakIns} the words $23\overline{2}\overline{3}1$ and $\overline{1}113\overline{3}$ are Knuth related:
	$\overline{1}113\overline{3} \stackrel{K2}{\sim}\overline{1}131\overline{3}\stackrel{K2}{\sim}\overline{1}13\overline{3}1
	\stackrel{K3}{\sim}2\overline{2}3\overline{3}1\stackrel{K1}{\sim}23\overline{2}\overline{3}1$.
	
	
	%
	
	\subsection{Crystal graphs in type $C$ and coplactic equivalence}\label{crystal}
	

	Crystals were originally  defined for quantum groups. Here we define them axiomatically associated to a root system $\Phi$ and a weight lattice $\Lambda$ \cite{BSch 17}.
	Let $V$ be an Euclidean space with inner product $\langle\cdot,\cdot\rangle$.
	Fix a root system $\Phi$ with simple roots $\{\alpha_i \mid i\in I\}$ where $I$ is an indexing set and a weight lattice $\Lambda\supseteq\Z\text{-span}\{\alpha_i \mid i\in I\}$. A \emph{Kashiwara crystal} of type $\Phi$ is a nonempty set $\mathfrak{B}$ together with maps \cite{BSch 17}:
	\begin{align*}
		e_i, f_i:\mathfrak{B}\rightarrow \mathfrak{B}\sqcup\{0\}\quad
		\varepsilon_i, \varphi_i:\mathfrak{B}\rightarrow \Z\sqcup\{-\infty\}\quad
		\text{wt}: \mathfrak{B}\rightarrow \Lambda
	\end{align*}
	where $i\in I$ and $0\notin \mathfrak{B}$ is an auxiliary element, satisfying the following conditions:
	\begin{enumerate}
		\item if $a, b \in \mathfrak{B}$ then $e_i(a)=b\Leftrightarrow f_i(b)=a$. In this case, we also have $\text{wt}(b)=\text{wt}(a)+\alpha_i$, $\varepsilon_i(b)=\varepsilon_i(a)-1$ and $\varphi_i(b)=\varphi_i(a)+1$;
		\item  for all $a \in \mathfrak{B}$, we have $\varphi_i(a)=\langle \text{wt}(a), \frac{2\alpha_i}{\langle \alpha_i, \alpha_i\rangle}\rangle+\varepsilon_i(a)$.
	\end{enumerate}

	The crystals we deal with are the ones of a $U_q(sp_{2n})$-module. They are seminormal \cite{BSch 17}, i.e., $\varphi_i(a)=\max\{k\in \Z\geq 0\mid f_i^k(a)\neq 0\}$ and  $\varepsilon_i(a)=\max\{k\in \Z\geq 0\mid e_i^k(a)\neq 0\}$.
	An element $u\in \mathfrak{B}$ such that $e_i(u)=0$ for all $i\in I $ is called  a \emph{highest weight element}. A \emph{lowest weight element} is an element $u\in \mathfrak{B}$ such that $f_i(u)=0$ for all $i\in I $.
	We associate with $\mathfrak{B}$ a coloured oriented graph with vertices in $\mathfrak{B}$ and edges labeled by $i\in I$: $b\overset{i}{\rightarrow} b'$ iff $b'=f_i(b)$, $i\in I$, $b, b'\in\mathfrak{B}$. This is the \emph{crystal graph} of $\mathfrak{B}$.
	
	If $\mathfrak{B}$ and $\mathfrak{C}$ are two seminormal crystals associated to the same root system, the \emph{tensor product} $\mathfrak{B}\otimes \mathfrak{C}$ is also a seminormal crystal. As a set, we will have the Cartesian product $\mathfrak{B}\times\mathfrak{C}$, where its elements are denoted by $b\otimes c$, $b\in\mathfrak{B}$ and $c\in\mathfrak{C}$, with $\text{wt}(b\otimes c)=\text{wt}(b)+\text{wt}(c)$, $f_i(b\otimes c)=\begin{cases}
	f_i(b)\otimes c\; \text{if}\; \varphi_i(c)\leq\varepsilon_i(b)\\
	b\otimes f_i(c)\; \text{if}\; \varphi_i(c)>\varepsilon_i(b)
	\end{cases}$, $e_i(b\otimes c)=\begin{cases}
	e_i(b)\otimes c\; \text{if}\; \varphi_i(c)<\varepsilon_i(b)\\
	b\otimes e_i(c)\; \text{if}\; \varphi_i(c)\geq\varepsilon_i(b)
	\end{cases}$. 
	
	If $\mathfrak{B}$ and $\mathfrak{C}$ are finite, $\varphi_i(b\otimes c)=\varphi_i(b)+\max(0, \varphi_i(c)-\varepsilon_i(b))$ and  $\varepsilon_i(b\otimes c)=\varepsilon_i(b)+\max(0, \varepsilon_i(b)-\varphi_i(c))$.

	In type $C_n$, we consider $\{\boldsymbol{e_i}\}_{i=1}^n$ the canonical basis of $\R^n$.
	The root system is $\Phi_C=\{\pm\boldsymbol{e_i}\pm\boldsymbol{e_j}\mid i< j \}\cup \{\pm2\boldsymbol{e_i}\}$ and the simple roots are $\alpha_i=\boldsymbol{e_i}-\boldsymbol{e_{i+1}}$, if $i\in [n-1]$,  $\alpha_n=2\boldsymbol{e_n}$. The weight lattice $\Z^n$ has dominant weights $\lambda=(\lambda_1\ge\dots\ge\lambda_n\ge 0)$. 
	
	In type $C_{n}$, the standard crystal is seminormal and has the following  crystal graph:
	$1\xrightarrow{1} 2\xrightarrow{2} \dots\xrightarrow{n-1} n\xrightarrow{n} \overline{n}\xrightarrow{n-1}\dots \xrightarrow{1}\overline{1}$  with set
	$\mathfrak{B}=[\pm n]$, $ \text{wt}(\boxed{i})=\bf e_i$, $ \text{wt}(\boxed{\overline{i}})=-\bf e_i$. The highest weight element is the word $1$, and the highest weight $\bf e_1$. We denote the crystal by $\mathfrak{B}^{\bf e_1}$.
	
	
	
	
	The crystal $\mathfrak{B}^{\bf e_1}$ is the crystal on the words of $[\pm n]^\ast$ of a sole letter. The tensor product of crystals allows us to define the crystal $G_n=\bigoplus\limits_{k\ge 0} {(\mathfrak{B}^{\bf e_1})}^{\otimes k}$ of all words in $[\pm n]^\ast$, where the vertex $w_1\otimes \dots \otimes w_k$ is identified with the word $w_1\dots w_k\in[\pm n]^\ast$. The action of the operators $e_i$ and $f_i$ is easily given by the signature rule \cite{NK 94,Lec 02,BSch 17}. We substitute each letter $w_j$ by $+$ if $w_j\in \{i, \overline{i+1}\}$ or by $-$ if $w_j\in \{i+1, \overline{i}\}$, and erase it in any other case. Then successively erase any pair $+-$ until all the remaining letters form a word that looks like $-^a +^b$. Then $\varphi_i(w)=b$ and $\varepsilon_i(w)=a$, $e_i$ acts on the letter associated to the rightmost unbracketed $-$ (i.e., not erased), whereas $f_i$ acts on the letter $w_j$ associated to the leftmost unbracketed $+$, $f_i(w_j)=\begin{cases} 
	i+1\,\,\text{if } w_j=i \text{ and } i\neq n \\
	\overline{i}\,\,\text{if }  w_j=\overline{i+1}\\
	\overline{n}\,\,  \text{if } w_j=i \text{ and } i=n
	\end{cases}$, and the other letters of $w$ are unchanged, and $e_i$ is the inverse map. If $b=0$ then $f_i(w)=0$ and if $a=0$ then $e_i(w)=0$. 
	
	\begin{example}
		Consider $w = \overline{2}31\overline{2}2\overline{1}$ and $i = 1$.
		Using the signature rule we rewrite $w$ as $+++--$. Now we erase pairs $+-$ as many times as possible, obtaining only $+$, that came from the first $\overline{2}$ in $w$. 
		
		Given that $f_1(\overline{2})=\overline{1}$, we have that $f_1(w)=\overline{1}31\overline{2}2\overline{1}$. Also, since there are no $-$ after eliminating all $+-$ pairs, we have that $e_1(w)=0$.
	\end{example}	
	
	The  crystal $G_n$, as a graph, is the union of connected components where each component has a unique highest weight word. Two connected components are isomorphic if and only if they have the same highest weight \cite{Kash 95}. 
	Two words in $[\pm n]^\ast$ are said to be crystal connected or coplactic equivalent if and only if they belong to the same connected component of $G_n$. This means that both words are obtained  from the same highest weight word, through a  sequence  of crystal operators $f_i$, or one is obtained from  another by some sequence of crystal operators $f_i$ and $e_j$, $i, j \in [n]$.  
	
	The connected components of $G_n$ are the coplactic classes in the Robinson-Schensted correspondence that identify words with the same oscillating tableau \cite[Proposition 5.2.1]{Lec 02}.
	Also, two words $w_1,w_2\in [\pm n]^\ast$ are Knuth equivalent if and only if they occur in the same place in two isomorphic connected components of $G_n$, that is, they are obtained from two highest words with the same weight through a same sequence of crystal operators \cite{Lec 02}. Crystal operators are coplactic and commute with the \emph{jeu de taquin}. The next proposition identifies all highest weight words of $G_n$.

	%
	%
	%
	
	\begin{proposition}
		Let $w$ be a word in the alphabet $[\pm n]$. Then $w$ is a highest weight word if and only if the weight of all its prefixes (including itself) is a partition. In this case, one has that $P(w)=K(\lambda)$ the Yamanouchi tableau of shape $\lambda$, where $\lambda$ is the weight of $w$.
	\end{proposition}
	\begin{proof} Part "if": We will prove the contrapositive of the statement. There is a $i$ such that $e_i(w)\neq 0$. Let $k$ be the position of the leftmost $-$ of the signature rule of $w$, and consider the prefix $w_k$ with the first $k$ letters. Since the $k$-th letter of $w$ had an unbracketed $-$ in the signature rule then the last letter of $w_k$ will also be an unbracketed $-$. Hence there are more $-$ than $+$ in the signature rule of $w_k$. Let $t_\alpha$ be the number of $\alpha$ in $w_k$. We have that $t_i+t_{\overline{i+1}}<t_{i+1}+t_{\overline{i}}\Leftrightarrow t_i-t_{\overline{i}}<t_{i+1}-t_{\overline{i+1}}$, hence the weight of $w_k$ is not a partition.
		
		Part "only if": We will once again prove the contrapositive of the statement. Let $w_k$ be a prefix such that its weight is not a partition. Hence there is $i\in [n]$ such that $t_i-t_{\overline{i}}<t_{i+1}-t_{\overline{i+1}}\Leftrightarrow t_i+t_{\overline{i+1}}<t_{i+1}+t_{\overline{i}}$, hence for this $i$ there will be more $-$ than $+$ in the signature rule of $w_k$. So in the first $k$ letters of $w$ there will be more $-$ than $+$, so there is an unbracketed $-$ in $w$, hence $e_i(w)\neq 0$.
		Note that the argument works even if $i=n$. In this case we need to assume $t_{n+1}=t_{\overline{n+i}}=0$.
		
		It follows from \cite[Proposition 3.2.6]{Lec 02} that the insertion of the highest word $w$ of weight $\lambda$ is $K(\lambda)$.
	\end{proof}
	
	Choose a word $w\in [\pm n]^\ast$ such that the shape of $P(w)$ is $\lambda$. If we replace every word of its coplactic class with its insertion tableau we obtain the crystal of tableaux $\mathfrak{B}^\lambda$ that has all KN tableaux of shape $\lambda$ on the alphabet $[\pm n]$. The crystal $\mathfrak{B}^\lambda$ does not depend on the initial choice of word $w$, as long as $P(w)$ has shape $\lambda$. \cite[Theorem 6.3.8]{Lec 02}.

	\section{Right and Left Keys and Demazure atoms in type $C$}\label{SecKeys}
	In this section, we  define type $C$ frank words on the alphabet $[\pm n]$ and use them to create the right and left key maps, that send KN tableaux to key tableaux in type $C$.  The main result of this section is the type $C$ version  \cite[Theorem 3.8]{LasSchu 90}, due to Lascoux and Schützenberger, which, using the right key map, gives  a combinatorial description of  Demazure atoms in type $C$.
	\subsection{Frank words in type $C$}
	Frank words were introduced in type $A$ by Lascoux and Schützenberger in \cite{LasSchu 90}.  We start by defining frank words in the alphabet $[\pm n]$.
	
	Given a ordered alphabet and a word on that alphabet, a column of the word is a maximal factor whose letters are strictly increasing. Hence, we can  decompose a word into columns, and such decomposition is unique.
	
	\begin{definition}
		Let $w$ be word on the alphabet $[\pm n]$.
		We say that $w$ is a\emph{ type $C$ frank word} if the lengths of its columns form a multiset equal to the multiset formed by the lengths of the columns of the tableau $P(w)$.
	\end{definition}
	\begin{example} In Example \ref{ExBakIns} we have that $P(23\overline{2}\overline{3}1)=P(\overline{1}113\overline{3})=\YT{0.17in}{}{
			{{1},{1},{\overline{1}}},
			{{3}},
			{{\overline{3}}}}$. Since $23\overline{2}\overline{3}1$ and $\overline{1}113\overline{3}$ have one column of length $3$ and two columns of length $1$, they are frank words. 
	\end{example}
	
	Given a frank word $w$, the number of letters of $w$ is the same as the number of cells of $P(w)$, hence the case $3$ of the Baker-Lecouvey insertion doesn't happen.
	
	\begin{proposition}
		Let $w$ be frank word on the alphabet $[\pm n]$. All columns of $w$ are admissible.
	\end{proposition} 
	\begin{proof}
		Suppose that the statement is false. So there is a factor of $w$ that is a non-admissible column with all of its proper factors admissible. Hence we can apply the Knuth relation $K5$, meaning that $w$ is Knuth related to a smaller word $w'$. But in this case, the number of letters of $w'$ is less then the number of cells of $P(w)=P(w')$, which is a contradiction.
	\end{proof}
	
	The following proposition is an extension of \cite[Proposition 7]{Ful 96} on SSYT to KN tableaux.
	
	\begin{proposition}\label{Fultonices}
		Let $T$ be a KN tableau of shape $\lambda$. Let $\mu/\nu$ be a skew diagram with same number of columns of each length as $T$. Then there is a unique KN skew tableau $S$ with shape $\mu/\nu$ that rectifies to $T$ and $cr(S)$ is a frank word. 
	\end{proposition}
	\begin{proof}
		If $T$ is a Yamanouchi tableau $K(\lambda)$ and $S\in \mathcal{KN}(\mu/\nu, n)$ rectifies to $K(\lambda)$,  then, since $S$ and $K(\lambda)$ have the same number of cells, all entries of $S$ are unbarred, hence $S$ is a semistandard skew tableau.  So, it follows from \cite[Proposition 7]{Ful 96} that $S$ exists and is unique. If $T$ is not a Yamanouchi tableau, note that $T$ is crystal connected to $K(\lambda)$ and from \cite[Theorem 6.3.8]{Lec 02} we have that the symplectic \textit{jeu de taquin} slides commutes with the action of the crystal operators. Consider $Y_\lambda'$ the only tableau on the skew-shape $\mu/\nu$ that rectifies to $Y_\lambda$, which exists due to \cite[Proposition 7]{Ful 96}. Since $S$ rectifies to $T$, which is crystal connected to $K(\lambda)$, and  $Y_\lambda'$ rectifies to $K(\lambda)$, $S$ is crystal connected to $Y_\lambda'$ and the path has same sequence of colours as the one from $T$ to $K(\lambda)$. Hence $S$ exists and is uniquely defined.
	\end{proof}
	
	\begin{corollary}\label{lastcolumn} Let $S$ be as in the previous proposition. 
		The last column of $S$ depends only on the length of that column.
	\end{corollary}
	\begin{proof}
		All other skew tableaux with given last column length can be found from a given one by playing the symplectic \textit{jeu de taquin} or its reverse in all columns except the last one. Note that $S$ has the same number of cells of the tableau obtained after rectifying, hence we can't lose cells when applying the symplectic \textit{jeu de taquin} or its reverse.
	\end{proof}
	
	Fixed a KN tableau $T$, consider the set of all possible last columns taken from skew tableaux with same number of columns of each length as $T$. Corollary \ref{lastcolumn} implies that this set has one element for each distinct column length of $T$. For every column $C$ in this set, consider the columns $rC$, its right column. The next proposition implies that this set of right columns is nested, if we see each column as the set formed by its elements.
	
	\begin{proposition}\label{nestedcolumns}
		Consider $T$ a two-column KN skew tableau $C_1C_2$ with an empty cell in the first column. Slide that cell once via symplectic \textit{jeu de taquin}, obtaining a two-column KN skew tableau $C_1'C_2'$ with an empty cell. Then $rC_2'\subseteq rC_2$.
	\end{proposition}
	\begin{proof}
		If the sliding was vertical then $C_2'=C_2$, hence $rC_2'=rC_2$. If the sliding was horizontal, Let $\beta$ be the number on the cell right of the empty cell on $spl(T)$. Recall $\Phi$, the function that takes an admissible column to the associated coadmissible column. 
		
		If $\beta=b$ is unbarred then $C_2'=\Phi^{-1}(\Phi(C_2)\setminus\{b\})$. In this case $\Phi(C_2')=\Phi(C_2)\setminus\{b\}$, hence $rC_2$ and $rC_2'$ have the same barred part. Consider $z_1<\dots<z_\ell$ the unbarred letters that appear on $C_2$ and not on $\Phi(C_2)$. When we take $b$ from $\Phi(C_2)$, if $\overline{b}\in \Phi(C_2)$ our set of letters $z_1<\dots<z_\ell$ will lose an element, giving the inclusion of the unbarred part of $C_2'$ in $C_2$; if $\overline{b}\not\in \Phi(C_2)$, then $b \in C_2$ and in $C_2'$ the least $z_i>b$ may reduce to $b$, and subsequent $z_j$ may reduce to $z_{j-1}$. Hence we have the inclusion of the unbarred part of $C_2'$ in $C_2$.
		
		If $\beta=\overline{b}$ is barred then $C_2'=C_2\setminus\{\overline{b}\}$. In this case $rC_2$ and $rC_2'$ have the same unbarred part.  Consider $\overline{t_1}>\dots>\overline{t_\ell}$ the barred letters that appear on $\Phi(C_2)$ and not on $C_2$. When we take $\overline{b}$ from $C_2$, if $b\in C_2$ our set of $\overline{t_1}>\dots>\overline{t_\ell}$ letters will lose an element, giving the inclusion of the barred part of $rC_2'$ in $rC_2$;  if $b\not\in C_2$, then $\overline{b} \in \Phi(C_2)$ and in $C_2'$ the least $\overline{z_i}>\overline{b}$ may reduce to $\overline{b}$, and subsequent bigger $\overline{z_j}$'s may reduce to $\overline{z_{j+1}}$. Hence we have the inclusion of the barred part of $\Phi(C_2')$ in $\Phi(C_2)$.
	\end{proof}
	
	This proposition defines a map that sends a KN tableau to a key tableau in type $C$, identified as the (symplectic) right key of a given KN tableau.
	
	\begin{theorem}[Right key map]\label{rightkeymap}
		Given a KN tableau $T$, we can replace each column with a column of the same size taken from the right columns of the last columns of all skew tableaux associated to it. We call this tableau the right key tableau of $T$ and denote it by $K_+(T)$.
	\end{theorem}
	\begin{proof}
		The previous proposition implies that the columns of $K_+(T)$ are nested and do not have symmetric entries. So, it is indeed a KN key tableau.
	\end{proof}
	\begin{remark}
		Recall the set up of Proposition \ref{Fultonices}. If the shape of $S$, $\mu/\nu$, is such that every two consecutive columns have at least one cell in the same row, then each column of $S$ is a column of the word $cr(S)$, hence $cr(S)$ is a frank word. Moreover, the columns of $S$ appear in reverse order in $cr(S)$. Therefore, given a KN tableau $T$, the columns of $K_+(T)$ can be also found as the right columns of the first columns of frank words associated to $T$.
		
		If $T$ is a SSYT then this right key map coincides with the one defined by Lascoux and Schützenberger in \cite{LasSchu 90}.
	\end{remark}
	\begin{example}\label{ExRightkey}
		The tableau $T=\YT{0.17in}{}{
			{{1},{3},{\overline{1}}},
			{{3},{\overline{3}}},
			{{\overline{3}}}}$ gives rise to six KN skew tableaux with same number of columns of each length as $T$, each one corresponding to a permutation of its column lengths, and each one is associated to its column reading, which is a frank word.
		
		
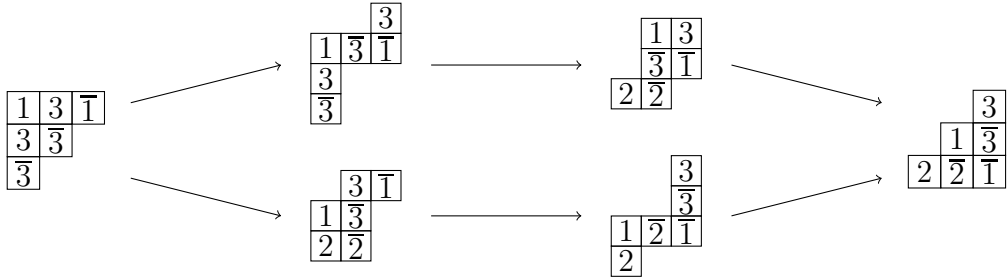
\begin{tikzpicture}
		\node at (0,0) {$\YT{0.17in}{}{
				{{1},{3},{\overline{1}}},
				{{3},{\overline{3}}},
				{{\overline{3}}}}$};	
		\node at (4,1) {$\begin{tikzpicture}[scale=.4, baseline={([yshift=-.8ex]current bounding box.center)}]
			\draw (0,0) rectangle +(1,1);
			\draw (0,1) rectangle +(1,1);
			\draw (0,2) rectangle +(1,1);
			\draw (1,2) rectangle +(1,1);
			\draw (2,2) rectangle +(1,1);
			\draw (2,3) rectangle +(1,1);
			\node at (.5,.5) {$\overline{3}$};
			\node at (.5,1.5) {$3$};
			\node at (.5,2.5) {$1$};
			\node at (1.5,2.5) {$\overline{3}$};
			\node at (2.5,2.5) {$\overline{1}$};
			\node at (2.5,3.5) {$3$};
			\end{tikzpicture}$};
		\node at (8,1) {$\begin{tikzpicture}[scale=.4, baseline={([yshift=-.8ex]current bounding box.center)}]
			\draw (0,0) rectangle +(1,1);
			\draw (1,0) rectangle +(1,1);
			\draw (1,1) rectangle +(1,1);
			\draw (1,2) rectangle +(1,1);
			\draw (2,1) rectangle +(1,1);
			\draw (2,2) rectangle +(1,1);
			\node at (.5,.5) {$2$};
			\node at (1.5,.5) {$\overline{2}$};
			\node at (1.5,1.5) {$\overline{3}$};
			\node at (1.5,2.5) {$1$};
			\node at (2.5,1.5) {$\overline{1}$};
			\node at (2.5,2.5) {$3$};
			\end{tikzpicture}$};
		\node at (4,-1) {$\begin{tikzpicture}[scale=.4, baseline={([yshift=-.8ex]current bounding box.center)}]
			\draw (0,0) rectangle +(1,1);
			\draw (0,1) rectangle +(1,1);
			\draw (1,0) rectangle +(1,1);
			\draw (1,1) rectangle +(1,1);
			\draw (1,2) rectangle +(1,1);
			\draw (2,2) rectangle +(1,1);
			\node at (.5,.5) {$2$};
			\node at (.5,1.5) {$1$};
			\node at (1.5,.5) {$\overline{2}$};
			\node at (1.5,1.5) {$\overline{3}$};
			\node at (1.5,2.5) {$3$};
			\node at (2.5,2.5) {$\overline{1}$};
			\end{tikzpicture}$};
		\node at (8,-1) {$$\begin{tikzpicture}[scale=.4, baseline={([yshift=-.8ex]current bounding box.center)}]
			\draw (0,0) rectangle +(1,1);
			\draw (0,1) rectangle +(1,1);
			\draw (1,1) rectangle +(1,1);
			\draw (2,1) rectangle +(1,1);
			\draw (2,2) rectangle +(1,1);
			\draw (2,3) rectangle +(1,1);
			\node at (.5,.5) {$2$};
			\node at (.5,1.5) {$1$};
			\node at (1.5,1.5) {$\overline{2}$};
			\node at (2.5,1.5) {$\overline{1}$};
			\node at (2.5,2.5) {$\overline{3}$};
			\node at (2.5,3.5) {$3$};
			\end{tikzpicture}$$};
		\node at (12,0) {$\SYT{0.17 in}{}{{{3}},{{\overline{3}},{1}},{{\overline{1}},{\overline{2}},{2}}}$};
		\draw [->] (1,.5) -- (3,1);
		\draw [->] (1,-.5) -- (3,-1);
		\draw [->] (5,1) -- (7,1);
		\draw [->] (5,-1) -- (7,-1);
		\draw [->] (9,1) -- (11,.5);
		\draw [->] (9,-1) -- (11,-.5);	
		\end{tikzpicture}
		
		The right key tableau associated to $T$ has as columns $r\!\YT{0.17in}{}{
			{{3}},
			{{\overline{3}}},
			{{\overline{1}}}}$, $r\!\YT{0.17in}{}{
			{{3}},
			{{\overline{1}}}}$ and $r\!\YT{0.17in}{}{
			{{\overline{1}}}}$. 
		Hence $K_+(T)=\YT{0.17in}{}{
			{{3},{3},{\overline{1}}},
			{{\overline{2}},{\overline{1}}},
			{{\overline{1}}}}$.
	\end{example}
	
	In the same spirit of the right key, we define the left key of a KN tableau. 
	Just like in Proposition \ref{nestedcolumns}, we can prove that the slides of the symplectic \textit{jeu de taquin} are effectively adding an entry to $\ell C_1$, i.e. $\ell C_1\subseteq \ell C_1'$, hence the left columns of the first columns of all skew tableaux with the same number of columns of each length as $T$ will be nested.
	
	So, if we replace each column of $T$ with a column of the same size taken from the left columns of the first columns of all skew tableaux associated to it we obtain the left key $K_-(T)$.
	
	\begin{example}In Example \ref{ExRightkey} we have that the left key of
		$T=\YT{0.17in}{}{
			{{1},{3},{\overline{1}}},
			{{3},{\overline{3}}},
			{{\overline{3}}}}$ has as columns
		
		$\ell\!\YT{0.17in}{}{
			{{1}},
			{{3}},
			{{\overline{3}}}}$, $\ell\!\YT{0.17in}{}{
			{{1}},
			{{2}}}$ and $\ell\!\YT{0.17in}{}{
			{{2}}}$. 
		Hence $K_-(T)=\YT{0.17in}{}{
			{{1},{1},{2}},
			{{2},{2}},
			{{\overline{3}}}}$.
	\end{example}

	\subsection{Demazure crystals and right key tableaux}
	
	Let $\lambda\in\Z^n$ be a partition and $v\in B_n\lambda $.
	We define $\mathfrak{U}(v)=\{T\in \mathcal{KN}(\lambda,n)\mid K_+(T)=K(v) \}$ the set of KN tableaux of $B^\lambda$ with right key $K(v)$.
	
	Given a subset $X$ of $\mathfrak{B}^\lambda$, consider the operator $\mathfrak{D}_i$ on $X$, with $i\in [n]$ defined by
	$\mathfrak{D}_iX=\{x\in \mathfrak{B}^\lambda\mid e_i^k(x)\in X \,\,\text{for some $k\geq 0$}\}$\cite{BSch 17}.
	If $v= \sigma \lambda$ where $\sigma =s_{i_{\ell(\sigma)}}\dots s_{i_1}\in B_n$ is a reduced word, we define the \emph{Demazure crystal} to be
	\begin{align}\label{Dem}
		\mathfrak{B}_v=\mathfrak{D}_{i_{\ell(\sigma)}}\dots\mathfrak{D}_{i_1}\{K(\lambda)\} .
	\end{align} 
	
	This definition is independent of the reduced word for $\sigma$ \cite[Theorem 13.5]{BSch 17}. In particular, when $\sigma$ is the longest element of $B_n$ we recover $\mathfrak{B}^{\lambda}$. Also this definition is independent of the coset representative of $\sigma W_\lambda$, that is, $\mathfrak{B}_{\sigma\lambda  }=\mathfrak{B}_{\sigma_v\lambda  }$. From \cite[Proposition 2.4.4]{BB05}, $\sigma$ uniquely factorizes  as $\sigma_v\sigma'$ where $\sigma'\in W_\lambda$ and $\ell(\sigma)=\ell(\sigma_v)+\ell(\sigma')$. From the signature rule, Subsection \ref{crystal}, if $ \sigma'=s_{j_{\ell(\sigma')}}\dots s_{j_1} \in W_\lambda$ is a reduced word, $\mathfrak{B}_{\sigma'\lambda }=\mathfrak{B}_{\lambda }=\mathfrak{D}_{j_{\ell(\sigma')}}\dots\mathfrak{D}_{j_1}\{K(\lambda)\} =\{K(\lambda)\}$ and we may write in \eqref{Dem} $\mathfrak{B}_{\sigma\lambda }=\mathfrak{B}_{v}$.
	
	From \cite[Proposition 2.5.1]{BB05}, if $\rho\leq\sigma$ in $B_n$ then $\rho_u\leq\sigma_v$ where $u=\rho\lambda$. Since $e_i^0(x)=x$, if $\rho\leq\sigma$ then  $\mathfrak{B}_{\rho\lambda }=\mathfrak{B}_{\rho_u\lambda }\subseteq\mathfrak{B}_{\sigma_v\lambda }=\mathfrak{B}_{ v}$.
	Thus we define the\emph{ Demazure  atom crystal} $\widehat{\mathfrak{B}}_{ v}$ to be
	\begin{equation}\label{atom}\widehat{\mathfrak{B}}_{ v}=\widehat{\mathfrak{B}}_{ \sigma\lambda }:=\mathfrak{B}_{\sigma_v\lambda  }\setminus\bigcup\limits_{\rho_u<\sigma_v }\mathfrak{B}_{\rho_u\lambda }=\mathfrak{B}_{v }\setminus\bigcup\limits_{u<v }\mathfrak{B}_{u}=\mathfrak{B}_{v }\setminus\bigcup\limits_{K(u)<K(v) }\mathfrak{B}_{u},\end{equation}
	where the two rightmost identities follow from Theorem \ref{keycoset}.

	\begin{lemma}\label{fi}
		Let $\sigma=s_i$  be a generator of $B_n$ and $C$ an admissible column such that $f_i(C)\neq 0$. Then $\text{\emph{wt}}(rC)=\text{\emph{wt}}(r(f_i(C)))$ or   $\text{\emph{wt}}(rC)=\sigma (\text{\emph{wt}}(r(f_i(C))))$.
	\end{lemma}
	\begin{proof}
		Let $i=n$. We can apply $f_i$ to $C$ if and only $n\in C$ and $\overline{n}\not\in C$. In this case $n\in rC$ and after applying $f_i$ we have $n\not\in C$ and $\overline{n}\in C$, hence $\overline{n}\in rC$. So $\text{wt}(rC)=s_n(\text{wt}(r(f_n(C))))$.
		
		Let $i<n$. We can apply $f_i$ to $C$, so we have $6$ cases to study:
		\begin{enumerate}
			\item $i\in C$, $i+1, \overline{i+1}, \overline{i}\not \in C$: In this case we have that $i+1\in f_i(C)$, $i, \overline{i+1}, \overline{i}\not \in f_i(C)$. Note that $\overline{i}\notin rC$ and $\overline{i+1}\notin r(f_i(C))$. If  $\overline{i+1}\not \in rC$ then  $\overline{i}\not \in r(f_i(C))$, hence $f_i$ swaps the weight of $i$ and $i+1$ from $(1,0)$ to $(0,1)$, respectively.
			If  $\overline{i+1} \in rC$ then $\overline{i} \in r(f_i(C))$, hence $f_i$ swaps the weight of $i$ and $i+1$ from $(1,-1)$ to $(-1,1)$.
			
			\item $i,\overline{i+1}\in C$, $i+1 , \overline{i}\not \in C$: In this case we have that $i+1, \overline{i+1}\in f_i(C)$, $i, \overline{i}\not \in f_i(C)$. Note that $i,\overline{i+1}\in rC$, $i+1, , \overline{i}\not \in rC$ and that $i+1, \overline{i}\in r(f_i(C))$, $i, \overline{i+1}\not \in r(f_i(C))$, and all other appearances in $rC$ are intact. Hence  $f_i$ swaps the weight of $i$ and $i+1$ from $(1,-1)$ to $(-1,1)$.
			
			\item $i+1,\overline{i+1}\in C$, $i , \overline{i}\not \in C$: In this case we have that $i+1, \overline{i}\in f_i(C)$, $i, \overline{i+1}\not \in f_i(C)$. Note that $i+1,\overline{i}\in rC$, $i, \overline{i+1}\not \in rC$ and that $i+1, \overline{i}\in r(f_i(C))$, $i, \overline{i+1}\not \in r(f_i(C))$, and all other appearances in $rC$ are intact. Hence  $f_i$ did nothing to weight of $rC$.
			
			\item $i, i+1,\overline{i+1}\in C$, $ \overline{i}\not \in C$: In this case we have that $i, i+1, \overline{i}\in f_i(C)$, $\overline{i+1}\not \in f_i(C)$. Note that $i,i+1\in rC$, $\overline{i+1} , \overline{i}\not \in rC$ and that $i, i+1\in r(f_i(C))$, $\overline{i+1}, \overline{i}\not \in r(f_i(C))$, and all other appearances in $rC$ are intact. Hence  $f_i$ did nothing to weight of $rC$.
			
			\item $i, \overline{i+1},\overline{i}\in C$, $i+1 \not \in C$: In this case we have that $i+1, \overline{i+1}, \overline{i}\in f_i(C)$, $i\not \in f_i(C)$. Note that $i, \overline{i+1}\in rC$, $i+1 , \overline{i}\not \in rC$ and that $i+1, \overline{i}\in r(f_i(C))$, $i, \overline{i+1} \not \in r(f_i(C))$, and all other appearances in $rC$ are intact. Hence $f_i$ swaps the weight of $i$ and $i+1$ from $(1,-1)$ to $(-1,1)$.
			
			\item $\overline{i+1}\in C$, $i,i+1, \overline{i}\not \in C$: In this case we have that $\overline{i}\in f_i(C)$, $i,i+1, \overline{i+1}\not \in f_i(C)$. Note that $i,i+1\not\in rC$ and $\overline{i+1} \in rC$.  If  $\overline{i} \in rC$ then we have  $i,i+1\not\in r(f_i(C))$ and $\overline{i+1},\overline{i} \in r(f_i(C))$, so $f_i$ did nothing to weight of $rC$. If  $\overline{i} \not\in rC$ then  $\overline{i+1}\not \in r(f_i(C))$ and  $\overline{i} \in r(f_i(C))$, hence $f_i$ swaps the weight of $i$ and $i+1$ from $(0,-1)$ to $(-1,0)$.
		\end{enumerate}
	\end{proof}
	\begin{remark}
		All the cases where the weight is preserved happen to have equal weight for $i$ or $i+1$ in $rC$ or we are in a column $C$ in which we can also apply $e_i$. If the weights for $i$ and $i+1$ in $rC$ swap, then if $rC$ the weight of $i$  is bigger (in the usual ordering) then the weight of $i+1$.
	\end{remark} Hence we have the following corollaries:
	
	\begin{corollary}\label{onde está fi de T}
		Let $T$ be a KN tableau and $i\in [n]$. If
		$K_+(T)=K(v)$, for some $v=(v_1,\dots, v_n)\in\Z^n$, then $K_+(f_i(T))=K(v) \text{ or } K_+(f_i(T))=K(s_iv).$
		Moreover, $K_+((T)f_i)=K(vs_i)$ only if $v_i>v_{i+1}$ (in the usual ordering of real numbers) and $1\leq i<n$, or, $v_i>0$ and $i=n$. 
	\end{corollary}
	\begin{proof}
		Consider a multiset of frank words $F$ such that the multiset of length of their first columns is the same of the multiset of lengths of columns of $T$. 
		
		If $K_+(f_i(T))=K_+(T)$ then we are done. Else there are two cases: $1\leq i<n$ and $i=n$.
		
		Consider $1\leq i<n$. 
		Since there is a change in the weight of the key tableau, we have that in at least one first column of words in $F$ weight of $i$ is bigger or equal than the weight of $i+1$. These first columns form a nested set without symmetric entries, hence in all first column of words in $F$ weight of $i$ is bigger or equal than the weight of $i+1$.

		Let $A$ be the subset of $F$ such that the weight of $i$ and $i+1$ in the right column of its first column is different and does not swap when we apply $f_i$ to the frank word.
		
		Consider $(a,b)$ the sum of weights of $i$ and $i+1$, respectively, of all right columns of first columns of words in $A$, and $(c,d)$ defined analogously to $F\setminus A$.  
		
		The weights of $i$ and $i+1$ in $K_+(T)$ is $(a,b)+(c,d)=(a+c,b+d)$ and the weights of $i$ and $i+1$ in $K_+(f_i(T))$ is $(a,b)+(d,c)=(a+d,b+c)$, and note that $(a+c,b+d)\in B_2(a+d,b+c)$, because $f_i$ doesn't change any other weight (Lemma \ref{fi}).
		
		Since in all first columns of $F$ weight of $i$ is bigger or equal than the weight of $i+1$, $a\geq0$ and $b\leq 0$, and they are equal when $A=\emptyset$, so $(a+c,b+d)= s_1(a+d,b+c)$, hence $\text{wt}(K_+(f_i(T)))=s_iv$. Hence we assume $a\neq b$.  If $c=d$ we have $\text{wt}(K_+(f_i(T)))=v$, hence $K_+(f_i(T))=K(v)=K_+(T)$, which is a contradiction.
		
		This implies that $(a+c,b+d)=\sigma(a+d,b+c)$ where $\sigma=\overline{12}$ or $\sigma=\overline{21}$.
		The first case implies that $a=\frac{-c-d}{2}=b$ and the second case implies  $c=\frac{-a-b}{2}=d$, hence there are not more possibilities for the weight of $K_+(f_i(T))$.
		
		The case $i=n$ is a simpler version of this one.
	\end{proof}

	\begin{corollary}\label{ei}
		Let $\sigma=s_i$  be a generator of $B_n$ and $C$ an admissible column. Then $\text{\emph{wt}}(rC)=\text{\emph{wt}}(r(e_i(C)))$ or   $\text{\emph{wt}}(rC)=\sigma(\text{\emph{wt}}(r(e_i(C))))$.
	\end{corollary}
	\begin{proof}
		Let $C'$ be $e_i(C)$. By Lemma \ref{fi} we have that  $\text{wt}(rC')=\text{wt}(r(f_i(C')))$ or   $\text{wt}(C')=\sigma(\text{wt}(r(f_i(C'))))$, so we have that   $\text{wt}(e_i(C))=\sigma(\text{wt}(rC))\Leftrightarrow \sigma(\text{wt}(e_i(C)))=\text{wt}(rC)$ or\linebreak $\text{wt}(r(e_i(C)))=\text{wt}(rC)$.
	\end{proof} 
	
	\begin{lemma}\label{eicheck}
		Let $i\in \left[n\right]$ and $C$ be an admissible column such that one of the following happens
		\begin{enumerate}
			\item $i<n$ and the weight of $i$ in $rC$ is less than the weight of $i+1$ in $rC$;
			\item $i=n$ and weight of $i$ is negative in $rC$,
		\end{enumerate}
		then we can apply $e_i$ to $C$ (in the sense $e_i(C)\neq 0$).
	\end{lemma}
	\begin{proof}
		If $i=n$ then $-n$ appears on $rC$ and $n$ does not. Since $n$ is the biggest unbarred letter of the alphabet we have that $-n$ also appears in $C$ and $n$ does not. Hence we can apply $e_n$ to $C$.
		
		If $i<n$ and the weight of $i$ in $rC$ is less than the weight of $i+1$ in $rC$ then the weight of both can be one of the following three options: $(0,1)$, $(-1,1)$, $(-1,0)$. Note that $rC$ does not have symmetric entries. So in the first two cases we have that $i+1$ exists in $rC$ and $i$ does not, hence $i+1$ exists in $C$ and $i$ does not, so we can apply $e_i$ to $C$.
		In the last case, we have that $\overline{i}$ exists in $rC$ and $i+1$ and $\overline{i+1}$ does not. Hence we have that $\overline{i}$ exists in $C$ and $i$ or $\overline{i+1}$ does not, so we can apply $e_i$ to $C$.
	\end{proof}
	
	The next theorem is the main theorem of this paper. It gives a description of a Demazure crystal atom in type $C$ using the right key map Theorem \ref{rightkeymap}. Lascoux and Schützenberger, in \cite[Theorem 3.8]{LasSchu 90}, proved the type $A$ version of this theorem, which consists in considering the case when $v\in \N^n$ and, consequently, $\sigma_v\in \mathfrak{S}_n$. For inductive reasoning, used in what follows, we recall the chain property on the set of minimal length coset representatives modulo $W_\lambda$ \cite[Theorem 2.5.5]{BB05}.
	
	\begin{theorem}\label{DemKey}
		Let $v\in B_n\lambda$.
		Then $\mathfrak{U}(v)=\widehat{\mathfrak{B}}_{v}$.
	\end{theorem}
	\begin{proof} Let $\rho$ be a minimal length coset representative modulo $W_\lambda$ such that $v=\rho\lambda$. We will proceed by induction on $\ell(\rho)$. 
		If $\ell(\rho)=0$ then $\rho=id$ and $v=\lambda$. In this case we have that $\widehat{\mathfrak{B}}_\lambda=\{K(\lambda)\}=\mathfrak{U}(\lambda)$.
		
		Let $\rho\geq 0$. Consider $\sigma=s_i$ a generator of $B_n$ such that $\sigma\rho>\rho$ and $\sigma\rho\lambda\neq\rho\lambda=v$, i.e., $\rho^{-1}\sigma\rho\notin W_\lambda$.
		Recall $e_i$, $\varepsilon_i$,  $f_i$ and $\phi_i$ from the definition of the crystal $\mathfrak{B}^\lambda$.
		If $T\in \widehat{\mathfrak{B}}_{\sigma\rho\lambda}$ then $T$ is obtained after applying $f_i$ (maybe more than once) to a tableau in $\widehat{\mathfrak{B}}_{\rho\lambda}$, which by inductive hypothesis exists in $\mathfrak{U}(v)$. By Corollary \ref{onde está fi de T}, if $f_i(T)\notin\mathfrak{U}(v)$ then $f_i(T)\in \mathfrak{U}(\sigma v)$. So it is enough to prove that given a tableau $T\in \mathfrak{U}(v)\cup \mathfrak{U}(\sigma v)$ then $e_i^{\varepsilon_i(T)}(T)\in\mathfrak{U}(v)$. 
		
		We have two different cases to consider: $i=n$ and $i<n$. 
		
		
		
		If $T\in \mathfrak{U}(\sigma v)$ then, if $i<n$, there exists a frank word of $T$ such that, if $V_1$ is its first column then $rV_1$ has less weight for $i$ than for $i+1$ (less in the usual ordering of real numbers); if $i=n$, there exists a frank word of $T$ such that, if $V_1$ is its first column then $rV_1$ has negative weight for $i$. Since we are in the column $rV_1$, if $i<n$, $i$ and $i+1$ can have weights $(0,1)$, $(-1,1)$ or $(-1,0)$ and if $i=n$ then $i$ has weight $-1$. Note that these are the exact conditions of Lemma \ref{eicheck}. In either case, due to Lemma \ref{eicheck}, we can applying $e_i$ enough times to the frank word associated until this no longer happens. This is true because we only need to look to $V_1$ to see if it changes after applying $e_i$ enough times to the frank word.
		In the signature rule we have that successive applications of $e_i$ changes the letters of a word from the end to the beginning, so, from the remark after Lemma \ref{fi}, the number of times that we need to apply $e_i$, in order to conditions of Lemma \ref{eicheck} do not hold for the first column, is $\varepsilon_i(T)$. So $K_+\left(e_i^{\varepsilon(T)}(T)\right)\neq K(\sigma v)$, hence, from Corollary \ref{ei}, we have that $e_i^{\varepsilon_i(T)}(T)\in\mathfrak{U}(v)$.

		If $T\in \mathfrak{U}(v)$ then $e_i^{\varepsilon_i(T)}(T)\in\mathfrak{U}(v)$ because if not, $e_i^{\varepsilon_i(T)}(T)$ will be in a Demazure crystal associated to $\rho'\in B_n$, with $\rho'<\rho$ such that $\sigma\rho'=\rho$. This cannot happen because in this case $\rho'=\sigma\rho<\rho$, which is a contradiction.
	\end{proof}

	\subsection{Combinatorial description of type $C$ Demazure characters and atoms}
	Given $v\in B_n\lambda $ define the Demazure character (or key polynomial), $\kappa_v$, and the Demazure atom in type $C$, $\widehat{\kappa}_v$, as the generating functions of the KN tableau weights in $\mathfrak{B}_v$ and $\widehat{\mathfrak{B}}_v$, respectively:  $\kappa_v=\sum\limits_{T\in \mathfrak{B}_{\sigma_v\lambda }}x^{\text{wt}\,T},\,
	\widehat{\kappa}_v=\sum\limits_{T\in \widehat{\mathfrak{B}}_{\sigma_v\lambda }}x^{\text{wt}\,T}$. Theorem \ref{DemKey} detects the KN tableaux in $\mathfrak{B}^\lambda$ contributing to the Demazure atom $\widehat{\kappa}_v$, $\widehat{\kappa}_v=\sum\limits_{\substack{K_+(T)=K(v)\\
			T\in \mathfrak{B}^\lambda }}x^{\text{wt}\,T}$.
	\begin{proposition}\label{Keysetal}
		Given $v\in B_n\lambda $, one has
		$\kappa_v=\sum\limits_{u\leq v}\widehat{\kappa}_u.$
	\end{proposition}
	\begin{proof}
		It is enough to prove that $\mathfrak{B}_v=\bigcup\limits_{u\leq v}\widehat{\mathfrak{B}}_u$, because $\kappa_v$ and $\widehat{\kappa}_u$ are the generating functions of the tableau weights in $\mathfrak{B}_v$ and $\widehat{\mathfrak{B}}_u$, respectively.
		Since $v=\sigma\lambda $, where $\sigma:=\sigma_v$, we can rewrite the identity as $\mathfrak{B}_{\sigma\lambda }=\bigcup\limits_{\rho\leq \sigma}\widehat{\mathfrak{B}}_{\rho\lambda }.$
		
		We will proceed by induction on $\ell(\sigma)$. If $\ell(\sigma)=0$ then the result follows because $\mathfrak{B}_{\lambda}=\widehat{\mathfrak{B}}_{\lambda}=\{K(\lambda)\}.$
		From \ref{atom}, $\widehat{\mathfrak{B}}_{\sigma\lambda }=\mathfrak{B}_{\sigma\lambda }\setminus\bigcup\limits_{\rho<\sigma}\mathfrak{B}_{\rho\lambda }$, and by inductive hypothesis, we have that $\mathfrak{B}_{\rho\lambda }=\bigcup\limits_{\rho'\leq \rho}\widehat{\mathfrak{B}}_{\rho'\lambda }.$
		Hence:
		\begin{align*}
			\widehat{\mathfrak{B}}_{\sigma\lambda }=\mathfrak{B}_{\sigma\lambda }\setminus\bigcup\limits_{\rho<\sigma}\mathfrak{B}_{\rho\lambda }=\mathfrak{B}_{\sigma\lambda }\setminus\bigcup\limits_{\rho<\sigma}\bigcup\limits_{\rho'\leq \rho}\widehat{\mathfrak{B}}_{\rho'\lambda }=\mathfrak{B}_{\sigma\lambda }\setminus\bigcup\limits_{\rho'<\sigma}\widehat{\mathfrak{B}}_{ \rho'\lambda}
		\end{align*}
	\end{proof}
	
	
	Proposition \ref{Keysetal}, the equivalence $u\leq v\Leftrightarrow K(u) \leq K(v)$, and Theorem \ref{DemKey}, allow to detect the KN tableaux contributing to a key polynomial in type $C$:
	
	\begin{align*}
		\kappa_v=\sum_{u\leq v}\widehat{\kappa}_u=\sum_{\substack{u\leq v\\T\in \mathfrak{U}(u)}}x^{\text{wt} T}=\sum_{\substack{K(u)\leq K(v)\\T\in \mathfrak{U}(u)}}x^{\text{wt} T}=\sum_{K_+(T)\leq K(v)}x^{\text{wt} T}.
	\end{align*}
	
	\begin{example}\label{cristal21}
		
		We start by looking to the crystal graph associated to the partition $\lambda=(2,1)$:
		\begin{multicols}{2}
			
			$
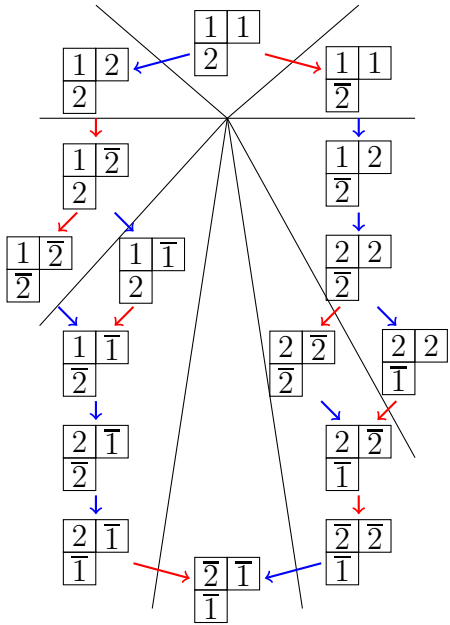
{
				\begin{tikzpicture}
				[scale=.5,auto=left]
				\node (n0) at (0,14.5) {$\YT{0.17 in}{}{{{1},{1}},{{2}}}$};
				\node (n1l) at (-3.5,13.5)  {$\YT{0.17 in}{}{{{1},{2}},{{2}}}$};
				\node (n1r) at (3.5,13.5)  {$\YT{0.17 in}{}{{{1},{1}},{{\overline{2}}}}$};
				\node (n2l) at (-3.5,11)  {$\YT{0.17 in}{}{{{1},{\overline{2}}},{{2}}}$};
				\node (n2r) at (3.5,11) {$\YT{0.17 in}{}{{{1},{2}},{{\overline{2}}}}$};
				\node (n3r) at (3.5,8.5)  {$\YT{0.17 in}{}{{{2},{2}},{{\overline{2}}}}$};
				\node (n4rr) at (5,6)  {$\YT{0.17 in}{}{{{2},{2}},{{\overline{1}}}}$};
				\node (n4r) at (2,6)  {$\YT{0.17 in}{}{{{2},{\overline{2}}},{{\overline{2}}}}$};
				\node (n5r) at (3.5,3.5)  {$\YT{0.17 in}{}{{{2},{\overline{2}}},{{\overline{1}}}}$};
				\node (n6r) at (3.5,1)  {$\YT{0.17 in}{}{{{\overline{2}},{\overline{2}}},{{\overline{1}}}}$};
				\node (n3ll) at (-5,8.5)  {$\YT{0.17 in}{}{{{1},{\overline{2}}},{{\overline{2}}}}$};
				\node (n4l) at (-3.5,6)  {$\YT{0.17 in}{}{{{1},{\overline{1}}},{{\overline{2}}}}$};
				\node (n3l) at (-2,8.5)  {$\YT{0.17 in}{}{{{1},{\overline{1}}},{{2}}}$};
				\node (n5l) at (-3.5,3.5)  {$\YT{0.17 in}{}{{{2},{\overline{1}}},{{\overline{2}}}}$};
				\node (n6l) at (-3.5,1)  {$\YT{0.17 in}{}{{{2},{\overline{1}}},{{\overline{1}}}}$};
				\node (n7) at (0,0)  {$\YT{0.17 in}{}{{{\overline{2}},{\overline{1}}},{{\overline{1}}}}$};
				
				\draw (0,12.5)--(-3.5,15.5);
				\draw (0,12.5)--(3.5,15.5);
				\draw (0,12.5)--(5,3.5);
				\draw (0,12.5)--(-5,7);
				\draw (0,12.5)--(-2,-.5);
				\draw (0,12.5)--(2,-.5);
				\draw (0,12.5)--(-5,12.5);
				\draw (0,12.5)--(5,12.5);
				\draw[->] [draw=red,  thick] (-2.5,0.7)--(-1,0.3);
				\draw[->] [draw=blue,  thick] (2.5,0.7)--(1,0.3);
				\draw[->] [draw=blue,  thick] (-3.5,2.5)--(-3.5,2);
				\draw[->] [draw=red,  thick] (3.5,2.5)--(3.5,2);
				\draw[->] [draw=blue,  thick] (-3.5,5)--(-3.5,4.5);
				\draw[->] [draw=blue,  thick] (2.5,5)--(3,4.5);
				\draw[->] [draw=red,  thick] (4.5,5)--(4,4.5);
				\draw[->] [draw=red,  thick] (-2.5,7.5)--(-3,7);
				\draw[->] [draw=blue,  thick] (-4.5,7.5)--(-4,7);
				\draw[<-] [draw=blue,  thick] (4.5,7)--(4,7.5);
				\draw[<-] [draw=red,  thick] (2.5,7)--(3,7.5);
				\draw[->] [draw=blue,  thick] (-3,10)--(-2.5,9.5);
				\draw[->] [draw=red,  thick] (-4,10)--(-4.5,9.5);
				\draw[->] [draw=blue,  thick] (3.5,10)--(3.5,9.5);
				\draw[->] [draw=red,  thick] (-3.5,12.5)--(-3.5,12);
				\draw[->] [draw=blue,  thick] (3.5,12.5)--(3.5,12);
				\draw[->] [draw=blue,  thick] (-1,14.2)--(-2.5,13.8);
				\draw[->] [draw=red,  thick] (1,14.2)--(2.5,13.8);
				\end{tikzpicture}
			}$	

			\noindent	The crystal is split into several parts. Each one of those parts is a Demazure atom and contains exactly one symplectic key tableau, so we can identify each part with the weight of that key tableau, which is a vector in the $B_2$-orbit of $(2,1)$. From Theorem \ref{DemKey} we have that all tableaux in the same part have the same right key.
			
			\noindent	One can check that, for example\\ $\mathfrak{U}((1,\overline{2}))=\left\{\YT{0.17 in}{}{{{1},{\overline{2}}},{{2}}}, \YT{0.17 in}{}{{{1},{\overline{2}}},{{\overline{2}}}}\right\}=\widehat{\mathfrak{B}}_{\lambda s_1s_2}$.
			
			\noindent	Also,
			
			\noindent $\!\mathfrak{B}_{(1,\overline{2})}\!\!=\!\!\{T\!\!\in\!\!\mathfrak{B}^\lambda\mid\!\! K_+(T)\!\leq\! K((1,\overline{2}))\}=\left\{\YT{0.17 in}{}{{{1},{1}},{{2}}}, \YT{0.17 in}{}{{{1},{2}},{{2}}}, \YT{0.17 in}{}{{{1},{1}},{{\overline{2}}}}, \YT{0.17 in}{}{{{1},{\overline{2}}},{{2}}}, \YT{0.17 in}{}{{{1},{\overline{2}}},{{\overline{2}}}}\right\}$.\end{multicols}
	\end{example}
	
	
	\section{Realization of the Lusztig involution in types $A$ and $C$}\label{SecLusztig}
	Let $\mathfrak{B}^\lambda$ be the crystal with set $\mathcal{KN}(\lambda, n)$ (respectively $SSYT(\lambda, n)$).
	\begin{definition}
		The Lusztig involution $L:\mathfrak{B}^\lambda\rightarrow\mathfrak{B}^\lambda$ is the only involution such that for all $i\in I$ ($I=[n-1]$ in type $A_{n-1}$ and $I=[n]$ in type $C_n$):
		\begin{enumerate}
			\item $\text{wt}(L(x))=\omega_0(\text{wt}(x))$,  where $\omega_0$ is the longest element of the Weyl group;
			\item $e_i(Lx)=L(f_{i'}(x))$ and $f_i(Lx)=L(e_{i'}(x))$ where $i'$ is such that $\omega_0(\alpha_i)=-\alpha_{i'}$ and $\alpha_i$ is the $i$-th simple root;
			\item $\varepsilon_i(Lx)=\varphi_{i'}(x)$ and $\varphi_i(Lx)=\varepsilon_{i'}(x)$.\end{enumerate}
	\end{definition}
	
	For type $A$ we have that $\omega_0$ is the reverse permutation and $i'=n-i$, and for type $C_n$ we have  $\omega_0=-\text{Id}$ and $i'=i$, where $\text{Id}$ is the identity map.
	In type $C_n$ the involution can be seen as flipping the crystal upside down.

	\begin{definition} \cite{BSch 17}
		Let $\mathfrak{C}$ be a  connected component  in the type $C_n$ crystal $G_n$. The dual crystal $\mathfrak{C}^\vee$  is the crystal obtained from $\mathfrak{C}$ after reversing the direction of all arrows. Also, the if $x\in\mathfrak{C}$, then for its correspondent in $\mathfrak{C}^\vee$, $x^\vee$, we have $\text{wt}(x)=-\text{wt}(x^\vee)$.
	\end{definition}
	In type $C$, since $i'=i$ and $\omega_0=-\text{Id}$, it follows from the definition that $\mathfrak{C}$ and $\mathfrak{C}^\vee$, as crystals in $G_n$, have the same highest weight. Therefore, they are isomorphic. In the case of $\mathfrak{B}^\lambda$, with set $\mathcal{KN}(\lambda, n)$, the Lusztig involution  is  a realization of the dual crystal. Hence the crystal $\mathfrak{B}^\lambda$ with set $\mathcal{KN}(\lambda, n)$ is self-dual. We shall see other realizations of the dual.

	\subsection{Evacuation algorithms}
	In type $A_{n-1}$, the Lusztig involution on the crystal with set $SSYT(\lambda, n)$ is known as Schützenberger involution or evacuation, ${Ev}$, and takes $T\!\in SSYT(\lambda, n)$ to $T^{Ev}\!\in\! SSYT(\lambda, n)$, whose weight is $\omega_0(\text{wt}\,T)$, where $\omega_0$ is the longest permutation of $\mathfrak{S}_n$, in the Bruhat order. Note that $\omega_0(\text{wt}\,T)$ is the vector $\text{wt}\,T$ in reverse order, i.e., $\omega_0(v_1, \dots, v_n)=(v_n,\dots, v_1)$.
	In type $C_n$ we will work with KN tableaux instead of SSYTs. Consider $T\in \mathcal{KN}(\lambda,n)$. In this case, $T^{{Ev}}\in \mathcal{KN}(\lambda,n)$ and $\text{wt}\, T=-\text{wt}\, T^{{Ev}}=\omega_0^C(\text{wt}\, T^{{Ev}})$, where $\omega_0^C$ is the longest permutation of $B_n$.
	The complement of a tableau or a word in types $A_{n-1}$ or $C_n$ consists in applying $\omega_0$ or $\omega_0^C$, respectively, to all its entries. In type $A_{n-1}$, it sends $i$ to $n+1-i$ for all $i\in[n]$, i.e., $\omega_0(i)=n+1-i$ and in type $C_n$ we have $\omega_0(i)=-i$. 
	Given a SSYT, there are several algorithms, due to Schützenberger, to obtain a SSYT with the same shape whose weight is its reverse. We recall some versions of them for which one is able to find analogues for KN tableaux.

	\begin{algorithm}\label{rsi}\indent
		\begin{enumerate}
			\item Define $w=cr(T)$.
			\item Define $w^\star$ the word obtained by complementing its letters and writing it backwards.
			\item $T^{Ev}:=P(w^\star)$.
		\end{enumerate}
	\end{algorithm}
	\begin{example}
		In type $A$, the tableau $T=\YT{0.145in}{}{
			{{1},{1},{2},{3}},
			{{2},{3},{3}},
			{{4}}}$ has reading $w=32313124$. Then $w^\star=13424232$, and the  column insertion of this word is
		$T^{Ev}=\YT{0.145in}{}{
			{{1},{2},{2},{3}},
			{{2},{4},{4}},
			{{3}}}$.

		In type $C$, consider the KN tableau $T=\YT{0.17in}{}{
			{{1},{3},{\overline{1}}},
			{{3},{\overline{3}}},
			{{\overline{3}}}}$. 
		Then, $w=cr(T)=\overline{1}3\overline{3}13\overline{3}$ and $w^\star=3\overline{3}\overline{1}3\overline{3}\overline{3}1$.
		So now we insert $w^\star$, obtaining the following sequence of tableaux:
		
		$\YT{0.17in}{}{
			{{{3}}}}
		\YT{0.17in}{}{
			{{3}},
			{{\overline{3}}}}
		\YT{0.17in}{}{
			{{3}},
			{{\overline{3}}},
			{{\overline{1}}}}
		\YT{0.17in}{}{
			{{2},{\overline{2}}},
			{{3}},
			{{\overline{1}}}}
		\YT{0.17in}{}{
			{{2},{\overline{2}}},
			{{3}, {\overline{1}}},
			{{\overline{3}}}}
		\YT{0.17in}{}{
			{{1},{2},{\overline{2}}},
			{{3}, {\overline{1}}},
			{{\overline{3}}}}=P(w^\star)$.
	\end{example}
	
	\begin{algorithm}\label{rcjt}\indent
		\begin{enumerate}
			\item Define $T^0:=\text{complement}(\text{$\pi$-rotate} (T))$.
			\item $T^{Ev}:=$ rectification of $T^0$.
		\end{enumerate}
	\end{algorithm}
	
	\begin{example}
		In type $A$, consider the tableau $T=\YT{0.145in}{}{
			{{1},{1},{2},{3}},
			{{2},{3},{3}},
			{{4}}}$. After $\pi$-rotation and complement we have the skew tableau
		$T^0=\SYT{0.145in}{}{
			{{1}},
			{{3},{2},{2}},
			{{4},{4},{3},{2}}}$ which, after rectification, gives the tableau
		$T^{Ev}=\YT{0.145in}{}{
			{{1},{2},{2},{3}},
			{{2},{4},{4}},
			{{3}}}$.
		
		In type $C$, consider the KN tableau $T=\YT{0.17in}{}{
			{{1},{3},{\overline{1}}},
			{{3},{\overline{3}}},
			{{\overline{3}}}}$. 	
		Then, $T_0=\SYT{0.17in}{}{
			{{3}},
			{{\overline{3}},{3}},
			{{\overline{1}},{\overline{3}}, {1}}}$.	
		So now we have to rectify this skew tableau obtaining	
		$T^{Ev}=\YT{0.17in}{}{
			{{1},{2},{\overline{2}}},
			{{3}, {\overline{1}}},
			{{\overline{3}}}}.$
	\end{example}
	Given a KN (SSYT) tableau $T$, the algorithm characterize $T^{Ev}$ as the unique KN tableau Knuth equivalent to $\text{wt}(T)^\star$ and coplactic equivalent do $T$.
	
	In both Cartan types we have that algorithms \ref{rsi} and \ref{rcjt} produce the same tableau since the column reading of $T^0$ is $w^\star$, $P(w^\star)=rect(T^0)=rect(w^\star)$, assuming that, in type $C_n$, $T^0$ is admissible. This can be concluded using the following lemma.
	\begin{lemma}\label{rotatestar}
		For type $C_n$, the split of a column $C$, $(\ell C, rC)$ is the rotation and complement of the split of the column $C^0=\text{complement}(\text{$\pi$-rotate}(C))$, $(\ell C^0, rC^0)$.
	\end{lemma}
	\begin{proof}
		Let's say that $(\ell C, rC)=\YT{0.17in}{}{
			{{A'},{A}},
			{{B},{B'}}}$ where $C=\YT{0.17in}{}{
			{{A}},
			{{B}}}$, $\ell C=\YT{0.17in}{}{
			{{A'}},
			{{B}}}$ and $rC=\YT{0.17in}{}{
			{{A}},
			{{B'}}}$,
		where $A$ and $A'$ are the unbarred letters of the columns $C$ and $\ell C$, respectively, and $B$ and $rB$ are the barred letters of $C$ and $rC$, respectively. Note that $\ell C$ and $C$ share the barred part and $C$ and $rC$ share the unbarred part.
		
		We have that $C^0=\YT{0.2in}{}{
			{{B^0}},
			{{A^0}}}$ and its split $(\ell C^0, rC^0)=\YT{0.2in}{}{
			{{B^{0'}},{B^0}},
			{{A^0},{A^{0'}}}}$.
		Ignoring bars and counting multiplicities, the letters that appear in $C$ and $C^0$ are the same. Hence $B^{0'}$ has the same letters as $B'$, but they appear unbarred, hence $B^{0 '}=B'^0$. The same happens with $A^{0'}$ and $A'^0$.
		Now it is easy to see that $(\ell C^0, rC^0)$ is obtained from $(\ell C, rC)$ rotating and complementing. In particular $(rC)^0=\ell C^0$ and $(\ell C)^0=rC^0$.	
	\end{proof}

	We now set the Cartan type to be $C$. Given a word $w\in[\pm n]^\ast$, we define the $w^\star$ like in the Algorithm \ref{rsi} and show that the map $^\star$ preserves Knuth equivalence. 
	
	\begin{theorem}\label{star preserves Knuth}
		Let $v, w \in [\pm n]^\ast$. Then $v\sim w$ if and only if $v^\star\sim w^\star$.
	\end{theorem}
	\begin{proof}
		It is enough to consider $v$ and $w$ only one Knuth relation apart, because all other cases are obtained by composing multiple Knuth relations. It is enough to consider each transformation applied in one direction, since the other direction is the same case, after swapping the roles of $v$ and $w$.
		\begin{enumerate}
			\item[K1] Consider $v=v_p\gamma \beta \alpha v_s$, with $\gamma <\alpha\leq \beta$ and $(\beta, \gamma)\neq(\overline{x},x)$, where $v_p$ is a prefix of $v$, $v_s$ is a suffix of $v$, and $\gamma \beta \alpha$ are three consecutive letters of $v$. Then, $v\stackrel{K1}{\sim}w=v_p\beta\gamma  \alpha v_s$.
			Note that $v^\star=v_s^\star \overline{\alpha} \overline{\beta} \overline{\gamma}  v_p^\star$ and $w^\star=v_s^\star \overline{\alpha} \overline{\gamma} \overline{\beta}  v_p^\star$, with $(\overline{\gamma}, \overline{\beta}) \neq (\overline{x}, x)$ and $\overline{\beta}\leq \overline{\alpha}<\overline{\gamma}$. 
			Hence $v^\star\stackrel{K2}{\sim}w^\star$, so they are Knuth related.

			\item[K2] Consider $v=v_p\alpha \beta \gamma  v_s$, with $\gamma \leq\alpha< \beta$ and $(\beta, \gamma)\neq(\overline{x},x)$, where $v_p$ is a prefix of $v$, $v_s$ is a suffix of $v$, and $\alpha \beta \gamma$ are three consecutive letters of $v$. Then, $v\stackrel{K2}{\sim}w=v_p\alpha\gamma  \beta v_s$.
			Note that $v^\star=v_s^\star \overline{\gamma} \overline{\beta} \overline{\alpha}  v_p^\star$ and $w^\star=v_s^\star \overline{\beta} \overline{\gamma} \overline{\alpha}  v_p^\star$, with $(\overline{\gamma}, \overline{\beta}) \neq (\overline{x}, x)$ and $\overline{\beta}< \overline{\alpha}\leq\overline{\gamma}$. 
			Hence $v^\star\stackrel{K1}{\sim}w^\star$, so they are Knuth related.

			\item[K3]  Consider $v=v_p (y+1) \overline{y+1} \beta  v_s$, with $y <\beta< \overline{y}$, where $v_p$ is a prefix of $v$, $v_s$ is a suffix of $v$, and $(y+1) \overline{y+1} \beta$ are three consecutive letters of $v$. Then, $v\stackrel{K3}{\sim}w=v_p\overline{y}y\beta v_s$.
			Note that $v^\star=v_s^\star \overline{\beta} (y+1) \overline{y+1}  v_p^\star$ and $w^\star=v_s^\star \overline{\beta}\overline{y}y  v_p^\star$, with  $y <\beta< \overline{y}$. 
			Hence $v^\star\stackrel{K4}{\sim}w^\star$, so they are Knuth related.

			\item[K4] Consider $v=v_p  \alpha \overline{x}x   v_s$, with $x <\alpha< \overline{x}$, where $v_p$ is a prefix of $v$, $v_s$ is a suffix of $v$, and $\alpha \overline{x}x$ are three consecutive letters of $v$. Then, $v\stackrel{K4}{\sim}w=v_p\alpha(x+1)\overline{x+1} v_s$.		
			Note that $v^\star=v_s^\star \overline{x} x \overline{\alpha}  v_p^\star$ and $w^\star=v_s^\star (x+1)\overline{x+1}\overline{\alpha}  v_p^\star$, with  $x <\alpha< \overline{x}$. 
			Hence $v^\star\stackrel{K3}{\sim}w^\star$, so they are Knuth related.
			
			\item[K5] Consider $w$ and $\{z, \overline{z}\}\in w$ such that $w\stackrel{K5}{\sim}w\setminus\{z,\overline{z}\}$. It is clear to see that a word $v$ breaks the $1CC$ at $z$ if and only if $v^\star$ breaks the $1CC$ at $z$. So, if $w$ is non admissible and all its factors are admissible then the same will happen to $w^\star$, because all of its factors are obtained after applying $^\star$ to a factor of $w$. So we have that $w^\star\stackrel{K5}{\sim}w^\star\setminus\{z,\overline{z}\}$.
		\end{enumerate}
		Hence the word operator $^\star$ preserves Knuth equivalence.	
	\end{proof}
	
	Consider a KN tableau $T$ with column reading $w$. The column reading of the tableau obtained after applying Algorithm \ref{rsi} to $T$ is Knuth-related to $w^\star$, because both give the same tableau if inserted. Since $^\star$ is an involution ($(w^\star)^\star=w$), if we apply the algorithm again we will get a tableau whose column reading, by the last theorem, is Knuth equivalent to $(w^\star)^\star=w$, hence we will have $T$ again. So Algorithm \ref{rsi} is an involution. Next we conclude that algorithms \ref{rsi} and \ref{rcjt} is a realization of the Lusztig involution for type $C$. 
	
	\begin{theorem} Let $w\in [\pm n]^\ast$.
		The connected component of the crystal $G_n$ that contains the word $w$ is isomorphic to the one that contains the word $w^\star$. Therefore $P(w)$ and $P(w^\star)$ have the same shape and weights of opposite sign. Moreover, the two crystals are dual of each other and the $^\star$ map is a realization of the dual crystal.
	\end{theorem}
	\begin{proof} Remember the crystal operators $e_i$ and $f_i$ from the definition of crystal. Note that $(f_i(w))^\star=e_i(w^\star)$, because in the signature rule applied to $w$ and $w^\star$, the distance of the leftmost unbracketed $+$ of $w$ to the beginning of the word is equal to the distance of the rightmost unbracketed $-$ of $w^\star$ to the end of this word. Hence, the letter that changes when applying $f_i$ to $w$ is the complement of the letter that changes when applying $e_i$ to $w^\star$, and the letter obtained on their position after applying the crystal operators are also complement of each other. Hence the crystal that contains the word $w^\star$ is the dual to the one that contains $w$. But the crystal that contains $w$ is self-dual, hence the crystals that contains any of the words are isomorphic. From \cite[Theorem 3.2.8]{Lec 02} $P(w)$ and $P(w^\star)$ have the same shape.
	\end{proof}

	\subsection{Right and left keys and Lusztig involution}
	
	The next result shows that the right and left key maps defined for KN tableaux anticommutes with the Lusztig involution. The evacuation of the right key of a tableau is the left key of the evacuation of the same tableau.
	
	\begin{proposition}\label{key&Lusz}
		Let $T$ be a KN tableau and ${Ev}$ the type $C$ Lusztig involution. Then
		$$K_+(T)^{Ev}=K_-(T^{Ev}).$$
	\end{proposition}
	\begin{proof}
		Since the tableaux $K_+(T)$ and $K_-(T^{Ev})$ are key tableaux, they are completely determined by their weights. Then we just need to prove that their weights are symmetric.
		
		Fix a column $C$ of $K_+(T)$. There is a frank word $w$, Knuth related to $cr(T)$, such that $C$ is the right column of the first column of $w$. Let's say the $w_k$ is the first column of $w$.
		From Proposition \ref{star preserves Knuth}, $w^\star$ is Knuth related to $cr(T)^\star$, hence $P(w^\star)=T^{Ev}$. Also note that the $w^\star$ has the same number of columns of each length as $w$, hence it is a frank word, and its last column is $w_k^\star$.
		Note that Lemma \ref{rotatestar} implies that if $v$ is an admissible column, then $l(v^\star)=(rv)^\star$. So we have that $l(w_k^\star)=(rw_k)^\star$ is a column of $K_-(T^{Ev})$.
		Therefore, for each column $C$ of $K_+(T)$ there is a column of $K_-(T^{Ev})$ whose weight is $\omega_0(\text{wt}(C))$, hence $K_+(T)$ and $K_-(T^{Ev})$ have symmetric weights.
	\end{proof}

	\section{Final Remarks}
	In \cite{Mas 09}, Mason   showed that Demazure atoms  are specializations of
	non-symmetric Macdonald polynomials of  type $A$ with $q = t = 0$. This allowed to use the shapes of semi-skyline augmented fillings,  in the combinatorial formula of non-symmetric Macdonald polynomials \cite{HHL 08}, which are in bijection with semi standard Young tableaux, to detect  the right keys. It would be interesting to obtain a similar object for a KN tableau in type $C$. For example, semi-skyline augmented fillings have been instrumental to obtain a RSK type bijective proof \cite{AzE 15} for the Lascoux non-symmetric Cauchy identity in type $A$ \cite{Las 03}. Such a generalization of skyline fillings for type $C$ could also lead to  a combinatorial formula for some specialization of nonsymmetric Macdonald polynomials in type $C$.
	
	In \cite{Willis 13}, Willis presents a direct algorithm to compute right keys in type $A$. It would be interesting to find something similar for type $C$. 
	
	In type $A$, key polynomials can also be described in terms of Kohnert diagrams \cite{Assaf 17, AssafSealer 19,Kohnert 90}. It would also be interesting to find an analogous description for type $C$.

	
	\subsection*{Acknowledgements}
	
	This work was partially supported by the Center for Mathematics of the  University of Coimbra - UID/MAT/00324/2019, funded by the Portuguese Government through FCT/MEC and co-funded by the European Regional Development Fund through the Partnership Agreement PT2020. It was also supported by FCT, through the grant PD/BD/142954/2018, under POCH funds, co-financed by the European Social Fund and Portuguese National Funds from MEC.
	
	I am grateful to C. Lenart, for his observations on the alcove path model, and to O. Azenhas, my Ph.D. advisor, for her help on the preparation of this paper. Also I wish to thank  the anonymous referees  for their careful reading and helpful comments on the manuscript.
	

\end{document}